\documentclass[10pt, a4paper]{article}
\usepackage[french, english]{babel}
\usepackage[utf8,latin1]{inputenc}
\usepackage{amsmath}
\usepackage{amssymb}
\usepackage{amsthm}
\usepackage{latexsym,tikz}
\usepackage{hyperref}
\usepackage{enumerate}
\usepackage[all]{xy}
\usepackage{mathrsfs}
\usepackage{dsfont}
\usepackage{tikz}
\usetikzlibrary{decorations.pathmorphing}
\usepackage{appendix}
\usetikzlibrary{decorations.pathreplacing}
\usepackage{pdfpages} 

\usepackage[nottoc]{tocbibind}

\newtheorem{thm}{Theorem}[section]
\newtheorem{cor}[thm]{Corollary}
\newtheorem{lem}[thm]{Lemma}
\newtheorem{prop}[thm]{Proposition}
\newtheorem{defn}[thm]{Definition}

\newtheorem{rem}[thm]{Remark}

\newtheorem{ex}[thm]{Example}

\def\bC{\mathbb{C}}

\def\det{\mathrm{det}}

\def\rG{\mathrm{G}}

\def\GL{\mathrm{GL}}

\def\cI{\mathcal{I}}

\def\Ind{\mathrm{Ind}}
\def\ind{\mathrm{ind}}

\def\cK{\mathcal{K}}

\def\rL{\mathrm{L}}

\def\rM{\mathrm{M}}

\def\rN{\mathrm{N}}

\def\cO{\mathcal{O}}

\def\cP{\mathcal{P}}

\def\bQ{\mathbb{Q}}

\def\obQ{\overline{\mathbb{Q}}}

\def\res{\mathrm{res}}

\def\SL{\mathrm{SL}}

\def\rU{\mathrm{U}}

\def\bZ{\mathbb{Z}}

\begin{document}





\hypersetup{							
pdfauthor = {Peiyi Cui},			
pdftitle = {representation mod l of SL_n},			
pdfkeywords = {Tag1, Tag2, Tag3, ...},	
}					
\title{Category decomposition of $\mathrm{Rep}_k(\mathrm{SL_n}(F))$}
\author{Peiyi Cui \footnote{peiyi.cuimath@gmail.com, Department of Mathematics, Oskar-Morgenstern-Platz 1, 1090 Vienna, Austria.}}

\date{\vspace{-2ex}}
\maketitle

 \begin{abstract}
 Let $F$ be a non-archimedean local field with residual characteristic $p$, and $k$ an algebraically closed field of characteristic $\ell\neq p$. We establish a category decomposition of $\mathrm{Rep}_k(\mathrm{SL}_n(F))$ according to the $\mathrm{GL}_n(F)$-inertially equivalent supercuspidal classes of $\mathrm{SL}_n(F)$, and we give a block decomposition of the supercuspidal sub-category of $\mathrm{Rep}_k(\mathrm{SL}_n(F))$. Finally we give an example to show that in general a block of $\mathrm{SL}_n(F)$ is not defined according to a unique inertially equivalent supercuspidal classes of $\mathrm{SL}_n(F)$, which is different from the case when $\ell=0$.
 \end{abstract}

\tableofcontents

\section{Introduction}
Let $F$ be a non-archimedean local field with residual characteristic $p$, and $k$ an algebraically closed field with characteristic $\ell$ different from $p$. We say $\rG$ is a $p$-adic group if it is the group of $F$-rational points of a connected reductive group $\mathbb{G}$ defined over $F$. Let $\mathrm{Rep}_k(\rG)$ be the category of smooth $k$-representations of $\rG$. In this article, we always denote by $\rM'$ a Levi subgroup of $\mathrm{SL}_n(F)$, and we study the category $\mathrm{Rep}_k(\rM')$. 

For arbitrary $p$-adic group $\rG$, we say that $\mathrm{Rep}_k(\rG)$ has a \textbf{category decomposition} according to an index set $\mathcal{A}$, if there exists an equivalence:

\begin{equation}
\label{0001}
\mathrm{Rep}_k(\rG)\cong\prod_{\alpha\in\mathcal{A}}\mathrm{Rep}(\rG)_{\alpha},
\end{equation}
where $\mathrm{Rep}_k(\rG)_{\alpha}$ are full sub-categories of $\mathrm{Rep}_k(\rG)$. The equivalence implies that:
\begin{itemize}
\item  Each object $\Pi\in\mathrm{Rep}_k(\rG)$ can be decomposed as a direct sum $\Pi\cong\oplus_{\alpha\in\mathcal{A}}\Pi_{\alpha}$, where $\Pi_{\alpha}\in\mathrm{Rep}_k(\rG)_{\alpha}$.
\item For $i=1,2$ and $\alpha_i\in\mathcal{A}$, if $\alpha_1\neq \alpha_2$, then $\mathrm{Hom}_{\rG}(\Pi_1,\Pi_2)=0$ for $\Pi_i\in\mathrm{Rep}_k(\rG)_{\alpha_i}$;
\end{itemize}
Furthermore, if
\begin{itemize}
\item for $\alpha\in\mathcal{A}$, there is no such decomposition for $\mathrm{Rep}_k(\rG)_{\alpha}$, we say that $\mathrm{Rep}_k(\rG)_{\alpha}$ is \textbf{non-split}. 
\end{itemize}
If $\mathrm{Rep}_k(\rG)_{\alpha}$ is non-split for each $\alpha\in\mathcal{A}$, we call this category decomposition a \textbf{block decomposition} of $\mathrm{Rep}_k(\rG)$, which means the finest category decomposition of $\mathrm{Rep}_k(\rG)$, and we call each $\mathrm{Rep}_k(\rG)_{\alpha}$ a \textbf{block} of $\mathrm{Rep}_k(\rG)$.

When $\ell=0$, a block decomposition of $\mathrm{Rep}_k(\rG)$ has been established according to $\mathcal{A}=\mathcal{SC}_{\rG}$, where $\mathcal{SC}_{\rG}$ is the set of $\rG$-inertially equivalent supercuspidal classes of $\rG$ (see Section \ref{section 002} for the definition). Let $[\rM,\pi]_{\rG}\in\mathcal{SC}_{\rG}$, where $(\rM,\pi)$ is a supercuspidal pair of $\rG$ (see Section \ref{section 002}). The sub-category $\mathrm{Rep}_k(\rG)_{[\rM,\pi]_{\rG}}$ consists of the objects whose irreducible sub-quotients have supercuspidal supports (see Section \ref{section 002}) in $[\rM,\pi]_{\rG}$.

When $\ell$ is positive, a block decomposition has been established when $\mathbb{G}$ is $\mathrm{GL}_n$ (\cite{Helm}) and its inner forms (\cite{SeSt}). For $\mathbb{G}=\mathrm{GL}_n$, the block decomposition is according to $\mathcal{SC}_{\rG}$ as well, which is the same as the case when $\ell=0$. It is worth noting that when we restrict the block decomposition in Equation \ref{0001} to the set of irreducible $k$-representations of $\rG$, the block decomposition according to $\mathcal{SC}_{\rG}$ requires that the supercuspidal support of an irreducible $k$-representation of $\rG$ belongs to a unique $\rG$-inertially equivalent supercuspidal class, which can be deduced from the uniqueness of supercuspidal support proved in \cite[\S V.4]{V2}  for $\mathrm{GL}_n(F)$. However the uniqueness of supercuspidal support is not true in general, in \cite{Da} an irreducible $k$-representation of $\mathrm{Sp}_8(F)$ that the supercuspidal support is not unique up to $\mathrm{Sp}_8(F)$-conjugation has been found. As for $\mathrm{SL}_n(F)$, the uniqueness of supercuspidal support holds true and has been proved in \cite{C2}, hence the block decomposition according to $\mathrm{SL}_n(F)$-inertially equivalent supercuspidal classes was expected. However in this article, we show that this is \textbf{not} always true by providing a counter-example in Section \ref{section 5}. 

\subsection{Main results} Now we describe the work in this article with more details. Let $\rM'$ be a Levi subgroup of $\mathrm{SL}_n(F)$, and $\rM$ be a Levi subgroup of $\mathrm{GL}_n(F)$ such that $\rM\cap\mathrm{SL}_n(F)=\rM'$, we establish a category decomposition of $\mathrm{Rep}_k(\rM')$ according to $\rM$-inertially equivalent supercuspidal classes $\mathcal{SC}_{\rM'}^{\rM}$ (see Section \ref{section 002} for the definitions), which is different from $\mathcal{SC}_{\rM'}$ the set of $\rM'$-inertially equivalent supercuspidal classes . In fact, let  $\rL$ be a Levi subgroup of $\rM$ and $\rL'=\rL\cap\rM'$ a Levi subgroup of $\rM'$, and let $\tau$ be an irreducible supercuspidal $k$-representation of $\rL$. Denote by $\mathcal{I}(\tau)$ the set of isomorphic classes of irreducible components of $\tau\vert_{\rL'}$. Let $\tau'\in\mathcal{I}(\tau)$, denote by $[\rL',\tau']_{\rM'}$ the $\rM'$-inertially equivalent supercuspidal class defined by $(\rM',\tau')$. The $\rM$-inertially equivalent supercuspidal class of $(\rL',\tau')$ is $\cup_{\gamma'\in\mathcal{I}(\tau)}[\rL',\gamma']_{\rM'}$, and we denote it by $[\rL',\tau']_{\rM}$.

\begin{thm}[Theorem \ref{thm 009}]
\label{thm 111}
Let $\mathcal{SC}_{\rM'}^{\rM}$ be the set of $\rM$-inertially equivalent supercuspidal classes of $\rM'$. There is a category decomposition of $\mathrm{Rep}_k(\rM')$ according to $\mathcal{SC}_{\rM'}^{\rM}$. 

In particular, let $[\rL',\tau']_{\rM}\in\mathcal{SC}_{\rM'}^{\rM}$, a $k$-representation $\Pi$ of $\rM'$ belongs to the full sub-category $\mathrm{Rep}_k(\rM')_{[\rL',\tau']_{\rM}}$, if and only if the supercuspidal support of each of its irreducible sub-quotients belongs to $[\rL',\tau']_{\rM}$.
\end{thm}

The above theorem gives a category decomposition
$$\mathrm{Rep}_k(\rM')\cong\mathrm{Rep}_k(\rM')_{\mathcal{SC}}\times\mathrm{Rep}_k(\rM')_{non-\mathcal{SC}},$$
where a $k$-representation $\Pi$ of $\rM'$ belongs to $\mathrm{Rep}_k(\rM')_{\mathcal{SC}}$ (resp. $\mathrm{Rep}_k(\rM')_{non-\mathcal{SC}}$) if each (resp. non) of its irreducible sub-quotients is supercuspidal. We call $\mathrm{Rep}_k(\rM')_{\mathcal{SC}}$ \textbf{the supercuspidal sub-category} of $\mathrm{Rep}_k(\rM')$. In Section \ref{section 5}, we establish a block decomposition of $\mathrm{Rep}_k(\rM')_{\mathcal{SC}}$.

Let  $\pi$ be an irreducible supercuspidal $k$-representation of $\rM$, and let $\mathcal{I}(\pi)$ be the set of isomorphic classes of irreducible components of $\pi\vert_{\rM'}$.  In Section \ref{section 5}, we introduce an equivalence relation $\sim$ on $\mathcal{I}(\pi)$.  For $\pi'\in\mathcal{I}(\pi)$, an irreducible supercuspidal $k$-representation of $\rM'$, let $(\pi',\sim)$ be the connected component of $\mathcal{I}(\pi)$ containing $\pi'$ under this equivalence relation, and let $[\pi',\sim]$ be the union of  $\rM'$-inertially equivalent supercuspidal classes of $\pi_j'\in(\pi',\sim)$. Denote by $\mathcal{SC}_{\rM',\sim}$ the set of pairs in the form of $[\pi',\sim]$. We establish a block decomposition of $\mathrm{Rep}_k(\rM')_{\mathcal{SC}}$:

\begin{thm}[Theorem \ref{thm 0030}]
\label{thm 222}
There is a block decomposition of $\mathrm{Rep}_k(\rM')_{\mathcal{SC}}$ according to $\mathcal{SC}_{\rM',\sim}$. In particular, let $[\pi',\sim]\in\mathcal{SC}_{\rM',\sim}$, a $k$-representation $\Pi$ belongs to $\mathrm{Rep}_k(\rM')_{[\pi',\sim]}$ if and only if each of its irreducible sub-quotients belongs to $[\pi',\sim]$.
\end{thm}

This article ends with Example \ref{ex 0029} of a $k$-representation in the supercuspidal sub-category of $\mathrm{Rep}_k(\mathrm{SL}_2(F))$ when $\ell=3$. In this example, we construct a finite length projective $k$-representation of $\mathrm{SL}_2(F)$ which is induced from a projective cover of a maximal simple supercuspidal $k$-type of depth zero. By using the theory of $k$-representations of finite $\mathrm{SL}_n$ group, we compute the irreducible sub-quotients of this projective cover, and we show that there exists two different supercuspidal $k$-representations of $\mathrm{SL}_2(F)$, which are not inertially equivalent, such that they belong to a same block. Or equivalently, this example shows that there exists an irreducible supercuspidal $k$-representation $\pi'$ of $\mathrm{SL}_2(F)$, such that $[\pi',\sim]$ is not a unique $\mathrm{SL}_2(F)$-inertially equivalent supercuspidal class, hence the equivalence relation defined on $\mathcal{I}(\pi)$ is non-trivial in general. This example shows that a block decomposition of $\mathrm{Rep}_k(\rG')$ (resp. $\mathrm{Rep}_k(\rM')$) according to $\rG'$-inertially equivalent supercuspidal classes $\mathcal{SC}_{\rG'}$ (resp. $\mathrm{Rep}_k(\rM')$-inertially equivalent supercuspidal classes $\mathcal{SC}_{\rM'}$) is not always possible in general.

\subsection{Structure of this paper}
The author is inspired by the method in \cite{Helm}. We use the theory of maximal simple $k$-types, which has been firstly established for $\bC$-representations of $\mathrm{GL}_n(F)$ in \cite{BuKuI} and generalised by the author to the cuspidal $k$-representations of $\rM'$ a Levi subgroup of $\mathrm{SL}_n(F)$ in \cite{C1}. In this article, we construct a family of projective objects defined from the projective cover of maximal simple $k$-types. In Section \ref{section 3}, we show that the projective cover of a maximal simple $k$-type of $\rM'$ is an indecomposable direct summand of the restriction of the projective cover of a maximal simple $k$-type of $\rM$. We apply the compact induction functor $\ind_{\rM'}^{\rM}$ to these projective objects and show their decomposition under the block decomposition of $\mathrm{Rep}_k(\rM)$. The above two parts leads to a family of injective objects verifying the conditions stated in Proposition \ref{prop 0011}, which gives the category decomposition in Theorem \ref{thm 111}.

Section \ref{section 5} concentrates on the supercuspidal sub-category of $\mathrm{Rep}_k(\rM')$, where $\rM'$ is a Levi subgroup of $\mathrm{SL}_n(F)$. We introduce an equivalence relation generated by putting all the irreducible sub-quotients of the projective cover of a maximal simple supercuspidal $k$-type of $\rM'$ into a same equivalent class. Let $\pi$ be an irreducible supercuspidal $k$-representation of $\rM$, the above equivalence relation on maximal simple supercuspidal $k$-types induces an equivalence relation on $\mathcal{I}(\pi)$, which is the equivalence relation $\sim$ needed in Theorem \ref{thm 222}.


There is a natural conjecture that a block decomposition of $\mathrm{Rep}_k(\rG')$ can be given according to the set of $\rG'$-conjugacy classes of elements in $\mathcal{SC}_{\rM',\sim}$ for all Levi subgroup $\rM'$, which involves the study of projective cover of maximal simple $k$-types (non-supercuspidal) of $\rM'$ and leads to a study of semisimple $k$-types of $\rG'$.

\section{Preliminary}
\subsection{Notations}
\label{section 002}
Let $F$ be a non-archimedean local field with residual characteristic equal to $p$. 

\begin{itemize}
\item $\mathfrak{o}_F$: the ring of integers of $F$, and $\mathfrak{p}_F$: the unique maximal ideal of $\mathfrak{o}_F$.
\item $k$: an algebraically closed field with characteristic $\ell\neq p$.
\item Let $K$ be a closed subgroup of a $p$-adic group $\rG$, then $\ind_{K}^{\rG}$: compact induction, $\Ind_{K}^{\rG}$: induction, $\res_{K}^{\rG}$: restriction.
\item Fix a split maximal torus of $\rG$, and $\rM$ be a Levi subgroup, then $i_{\rM}^{\rG}, r_{\rM}^{\rG}$: normalised Parabolic induction and normalised Parabolic restriction.
\item Denote by $\delta_{\rG}$ the module character of $\rG$.
\end{itemize}

In this article, without specified we always denote by $\rG$ the group of $F$-rational points of $\mathrm{GL}_n$ and by $\rG'$ the group of $F$-rational points of $\mathrm{SL}_n$. Let $\iota$ be the canonical embedding from $\rG'$ to $\rG$, which induces an isomorphism between the Weyl group of $\rG'$ and $\rG$, hence gives a bijection from the set of Levi subgroups of $\rG'$ to those of $\rG$. In particular, suppose $\rM$ is a Levi subgroup of $\rG$, we always denote by $\rM'$ the Levi subgroup $\rM\cap\rG'$ of $\rG'$. We say an irreducible $k$-representation $\pi$ of a $p$-adic group $\rG$ is \textbf{cuspidal}, if $r_{\rM}^{\rG}\pi$ is zero for every proper Levi subgroup $\rM$; we say $\pi$ is \textbf{supercuspidal} if it does not appear as a sub-quotient of $i_{\rM}^{\rG}\tau$ for each proper Levi subgroup $\rM$ and its irreducible representation $\tau$.

Let $\pi$ be an irreducible $k$-representation of $\rG$, its restriction $\pi\vert_{\rG'}$ is semisimple with finite length, and each irreducible $k$-representation $\pi'$ of $\rG'$ appears as a direct component of $\pi\vert_{\rG'}$. A pair $(\rM,\tau)$ is called a \textbf{cuspidal} (resp. \textbf{supercuspidal}) \textbf{pair} if $\rM$ is a Levi subgroup and $\tau$ is an irreducible cuspidal (resp. supercuspidal) of $\rM$. Let $(\rM_1',\tau_1'), (\rM_2',\tau_2')$ be two cuspidal pairs of $\rG'$ and $K$ be a subgroup of $\rG$, we say they are $K$-\textbf{inertially equivalent}, if there exists an element $g\in K$ such that $g(\rM_1')=\rM_2'$ and there exists an unramified $k$-quasicharacter $\theta$ of $F^{\times}$ such that $g(\tau_1')\cong\tau_2'\otimes\theta$. Denote by $[\rM',\tau']_{K}$ the $K$-inertially equivalent class defined by $(\rM',\tau')$, and we call it a $K$-inertially equivalent supercuspidal (resp. cuspidal) class if $(\rM',\tau')$ is a supercuspidal (resp. cuspidal) pair. A same definition of $[\rM,\tau]_{\rG}$ is applied for cuspidal pairs of $\rG$. We always abbreviate $[\rM',\tau']_{\rG'}$ as $[\rM',\tau']$, and abbreviate $[\rM,\tau]_{\rG}$ as $[\rM,\tau]$.   

We say that a cuspidal (resp. supercuspidal) pair $(\rM,\tau)$ belongs to the \textbf{cuspidal} (resp. \textbf{supercuspidal}) \textbf{support} of $\pi$, if $\pi$ appears as a sub-representation or a quotient-representation (resp. sub-quotient representation) of $i_{\rM}^{\rG}\tau$. When $\pi$ is an irreducible $k$-representation of $\rG$ (resp. $\rG'$), its supercuspidal support as well as its cuspidal support is \textbf{unique} up to $\rG$(resp. $\rG'$)-conjugation.

To decompose $\mathrm{Rep}_k(\rG')$ as a direct product of a familly of full-subcategories, we constructs a familly of injective objects and follow the method as below, which is the same strategy as in \cite[Proposition 2.4]{Helm}. We state it here for convenient reason.
\begin{prop}
\label{prop 0011}
Let $\mathcal{I}_1,\mathcal{I}_2$ be two injective objects in $\mathrm{Rep}_k(\rG')$, and denote by $\mathcal{S}_1,\mathcal{S}_2$ the sets of irreducible $k$-representations of $\rG'$ which appears as a sub-quotient of $\mathcal{I}_1$ and $\mathcal{I}_2$ respectively. Suppose the following conditions are verified:
\begin{itemize}
\item an object in $\mathcal{S}_1$ can be embedded into $\mathcal{I}_1$;
\item an object in $\mathcal{S}_1$ does not belong to $\mathcal{S}_2$ up to isomorphism;
\item an irreducible $k$-representation of $\rG'$, which does not belong to $\mathcal{S}_1$ up to isomorphism, can be embedded into $\mathcal{I}_2$.
\end{itemize}
Then $\mathrm{Rep}_k(\rG')$ can be decomposed as a direct product of two full subcategories $\mathrm{R}_1$ and $\mathrm{R}_2$, such that
\begin{itemize}
\item every object $\Pi\in\mathrm{Rep}_k(\rG')$ is isomorphic to a direct sum $\pi_1\oplus\pi_2$, where each irreducible sub-quotient of $\pi_1$ belongs to $\mathcal{S}_1$ and each irreducible sub-quotient of $\pi_2$ belongs to $\mathcal{S}_2$;
\item every object in $\mathrm{R}_1$ has an injective resolution by direct sums of copies of $\mathcal{I}_1$, and every object in $\mathrm{R}_2$ has an injective resolution by direct sums of copies of $\mathcal{I}_2$ (copies means direct product by itself).
\end{itemize}
\end{prop}

\begin{rem}[Projective version]
\label{rem 0029}
Let $\mathcal{P}_1,\mathcal{P}_2$ be two projective objects in $\mathrm{Rep}_k(\rG')$, and denote by $\mathcal{S}_1,\mathcal{S}_2$ the sets of irreducible $k$-representations of $\rG'$ which appears as a sub-quotient of $\mathcal{P}_1$ and $\mathcal{P}_2$ respectively. Suppose the following conditions are verified:
\begin{itemize}
\item an object in $\mathcal{S}_1$ is a quotient of $\mathcal{P}_1$;
\item an object in $\mathcal{S}_1$ does not belong to $\mathcal{S}_2$ up to isomorphism;
\item an irreducible $k$-representation of $\rG'$, which does not belong to $\mathcal{S}_1$ up to isomorphism, can be realised as a quotient of $\mathcal{P}_2$.
\end{itemize}
Then $\mathrm{Rep}_k(\rG')$ can be decomposed as a direct product of two full subcategories $\mathrm{R}_1$ and $\mathrm{R}_2$, such that every object $\Pi\in\mathrm{Rep}_k(\rG')$ is isomorphic to a direct sum $\pi_1\oplus\pi_2$, where each irreducible sub-quotient of $\pi_1$ belongs to $\mathcal{S}$ and each irreducible sub-quotient of $\pi_2$ belongs to $\mathcal{S}_2$.
\end{rem}

The proof of Remark \ref{rem 0029} is in the same manner as in Proposition 2.4 of \cite{Helm} by changing injective objects to projective objects as suggested in Remark 2.5 of \cite{Helm}.

\subsection{Maximal simple $k$-types of $\rM'$}
\label{section 1.0}
In this section, we recall notations and definitions in the theory of maximal simple $k$-types of Levi subgroups $\rM'$ of $\rG'$ which has been studies in \cite{C1}. It requires the theory of maximal simple $k$-types of $\rG$ which has been established in \cite{BuKuII} for complex case. The later is related to modulo $\ell$ maximal simple types in \cite[\S III]{V1} by considering the reduction modulo $\ell$, while \cite{MS} gives a more intrinsic description. We state some useful properties which will be needed for the further use.

A \textbf{maximal simple $k$-type} of $\rG$ is a pair $(J,\lambda)$, where $J$ is an open compact subgroup of $\rG$ and $\lambda$ is an irreducible $k$-representation of $J$. We have a groups inclusion:
$$H^1\subset J^1 \subset J,$$
where $J^1$ is a normal pro-$p$ open subgroup $J^1$ of $J$, such that the quotient $J\slash J^1$ is isomorphic to $\mathrm{GL}_m(\mathbb{F}_q)$, where $\mathbb{F}_q$ is a field extension of the residue field of $F$, and $H^1$ is open. The $k$-representation $\lambda$ is a tensor product $\kappa\otimes\sigma$, where $\kappa$ is irreducible whose restriction to $H^1$ is a multiple of a $k$-character, and $\sigma$ is inflated from a cuspidal $k$-representation of $J\slash J^1$. By \cite[\S III 4.25]{V1} or \cite[Proposition 3.1]{MS}, for an irreducible $k$-cuspidal representation $\pi$ of $\rG$, there exists a maximal simple $k$-type $(J,\lambda)$, and compact modulo centre subgroup $K$ and an irreducible representation $\Lambda$ of $K$, where  $J$ is the unique largest compact open subgroup of $K$ and $\Lambda$ is an extension of $\lambda$, such that $\pi\cong\ind_{K}^{\rG}\Lambda$. Since a Levi subgroup of $\rG$ is a tensor product of $\mathrm{GL}$-groups in lower rank, so we can define maximal simple $k$-types $(J_{\rM},\lambda_{\rM})$ and obtain the same property for cuspidal $k$-representations of $\rM$ as above.

For the reason that a Levi subgroups $\rM'$ of $\rG'$ is not a product of $\mathrm{SL}$-groups of lower rank, so it is not sufficient to consider only the maximal simple $k$-types of $\rG'$. Let $(J_{\rM},\lambda_{\rM})$ be a maximal simple $k$-type of $\rM$, the group of projective normaliser $\tilde{J}_{\rM}$ contains $J_{\rM}$ as a normal subgroup, which is defined in \cite[2.15]{C1} and \cite[2.2]{BuKuII}. In particular, for any $g\in\tilde{J}_{\rM}$, the conjugate $g(\lambda_{\rM})\cong\lambda_{\rM}\otimes\chi$, where $\chi$ is a $k$-quasicharacter of $F^{\times}$. As in \cite[2.48]{C1}, a \textbf{maximal simple $k$-type} of $\rM'$ is a pair in the form of $(\tilde{J}_{\rM}',\tilde{\lambda}_{\rM}')$, where $\tilde{\lambda}_{\rM}'$ is an irreducible direct component of $(\ind_{J_{\rM}}^{\tilde{J}_{\rM}}\lambda_{\rM})\vert_{\tilde{J}_{\rM}'}$, and we denote by $\tilde{\lambda}_{\rM}$ as $\ind_{J_{\rM}}^{\tilde{J}_{\rM}}\lambda_{\rM}$, which is irreducible as proved in \cite[Theorem 2.47]{C1}. For any irreducible cuspidal $k$-representation $\pi'$ of $\rM'$, there exists an irreducible cuspidal $k$-representation $\pi$ of $\rM$ such that $\pi'$ is a direct component of $\pi\vert_{\rM'}$. Let $(J_{\rM},\lambda_{\rM})$ be a maximal simple $k$-types contained in $\pi$, then there exists a maximal simple $k$-type $(\tilde{J}_{\rM},\tilde{\lambda}_{\rM}')$ as well as an open compact modulo centre subgroup $N_{\rM'}(\tilde{\lambda}_{\rM}')$, the normaliser group of $\tilde{\lambda}_{\rM}'$ in $\rM'$, containing $\tilde{J}_{\rM}'$ as its largest open compact subgroup, as well as an extension $\Lambda_{\rM'}$ of $\tilde{\lambda}_{\rM}'$ to $N_{\rM'}(\tilde{\lambda}_{\rM}')$, such that $\pi'\cong\ind_{N_{\rM'}(\tilde{\lambda}_{\rM}')}^{\rM'}\Lambda_{\rM'}$. We call $(N_{\rM'}(\tilde{\lambda}_{\rM}'),\Lambda_{\rM}')$ an \textbf{extended maximal simple $k$-type}.

\begin{prop}[C.]
Let $\pi'$ be an irreducible cuspidal $k$-representation of $\rM'$, there exits a cuspidal $k$-representation $\pi$ of $\rM$, such that $\pi'$ is a direct component of $\pi\vert_{\rM'}$. Then $\pi'$ is supercuspidal if and only if $\pi$ is supercuspidal. When $\pi$ is supercuspidal, we call a $k$-type  $(J_{\rM},\lambda_{\rM})$ (resp. $(\tilde{J}_{\rM}',\tilde{\lambda}_{\rM}'))$ contained in $\pi$ (resp. $\pi'$) a \textbf{maximal simple supercuspidal $k$-type}. 
\end{prop}

Let $K_1,K_2$ be two open subgroups of $\rM'$ and $\rho_1,\rho_2$ two irreducible $k$-representations of $K_1,K_2$ respectively. We say that $\rho_1$ is \textbf{weakly intertwined} with $\rho_2$ in $\rM'$, if there exists an element $m\in\rM'$ such that $\rho_1$ is isomorphic to a sub-quotient of $\ind_{K_1\cap m(K_2)}^{K_1}\res_{K_1\cap m(K_2)}^{m(K_2)} m(\rho_2)$. Recall that:
\begin{prop}[C.]
\begin{enumerate}
\item $\tilde{J}_{\rM}=\tilde{J}_{\rM}'J_{\rM}$;
\item let $(\tilde{J}_{\rM,1}',\tilde{\lambda}_{\rM,1}')$ and $(\tilde{J}_{\rM,2}',\tilde{\lambda}_{\rM,2}')$ be two maximal simple $k$-types of $\rM'$, they are weakly intertwined in $\rM'$ if and only if they are $\rM'$-conjugate.
\end{enumerate}
\end{prop}

\section{Category decomposition}
In this section, to simplify the notations, we denote by $\rG$ a Levi subgroup of $\mathrm{GL}_n(F)$ and $\rG'=\rG\cap\mathrm{SL}_n(F)$, which is a Levi subgroup of $\mathrm{SL}_n(F)$. Let $\rM$ be a Levi subgroup of $\rG$, we denote by $\rM'=\rM\cap\rG'$ a Levi subgroup of $\rG'$, and let $K$ be an open subgroup of $\rG$, we always denote by $K'=K\cap\rG'$. If $\pi$ is an irreducible $k$-representation of $K$, then $\pi'$ is one of its irreducible summand of $\pi\vert_{K'}$.

\subsection{Projective objects}
\label{section 3}
In this section, we will follow the strategy as in \cite{Helm} to construction some projective objects of $\mathrm{Rep}_k(\rG')$. We study first the projective cover of maximal simple $k$-types of Levi subgroups $\rM'$, then we consider their induced representations. Proposition \ref{prop 008} and Corollary \ref{cor 0010} give the relation between these projective objects and irreducible $k$-representations whose cuspidal support is given by the corresponding maximal simple $k$-type. The later properties will be used in Section \ref{section 4}.

Let $(J_{\rM},\lambda_{\rM})$ be a maximal simple $k$-type of $\rM$, and $\tilde{J}_{\rM}$ be the group of projective normaliser of $(J_{\rM},\lambda_{\rM})$ (see Section \ref{section 1.0}). Write $\lambda_{\rM}$ as $\kappa_{\rM}\otimes\sigma_{\rM}$. Let $\mathcal{P}_{\lambda_{\rM}}$ be the projective cover of $\lambda_{\rM}$, from \cite[Lemma 4.8]{Helm} we know that $\mathcal{P}_{\lambda_{\rM}}$ is isomorphic to $\cP_{\sigma_{\rM}}\otimes\kappa_{\rM}$, where $\mathcal{P}_{\sigma_{\rM}}$ is the projective cover of $\sigma_{\rM}$. Denote by $\tilde{\lambda}_{\rM}$ the irreducible $k$-representation $\ind_{J_{\rM}}^{\tilde{J}_{\rM}}\lambda_{\rM}$. Let $(\tilde{J}_{\rM}',\tilde{\lambda}_{\rM}')$ be a maximal simple $k$-type of $\rM'$ defined from $(J_{\rM},\lambda_{\rM})$ as in Section \ref{section 1.0}. Since $\cP_{\lambda_{\rM}}$ has finite length, we have $\mathcal{P}_{\lambda_{\rM}}\vert_{J_{\rM}'}=\oplus_{i=1}^s\mathcal{P}_i$, where $\mathcal{P}_i$ is an indecomposable projective $k$-representation of $J_{\rM}'$ for each $i$.

\begin{rem}
\label{rem 0014}
The projective cover $\mathcal{P}_{\sigma_{\rM}}$ is given by the theory of $k$-representations of finite general linear groups. When $\sigma_{\rM}$ is inflated from a supercuspidal $k$-representation of $\rM$, which means $(J_{\rM},\lambda_{\rM})$ is a maximal simple supercuspidal $k$-type of $\rM$, according to the construction of $\mathcal{P}_{\sigma_{\rM}}$ (in Lemma 5.11 of \cite{Geck} or see Corollary 3.5 of \cite{C2} ) as well as Deligne-Lusztig theory, we conclude that the irreducible subquotients of $\mathcal{P}_{\sigma_{\rM}}$ are isomorphic to $\sigma_{\rM}$.
\end{rem}


Let $\pi$ be an irreducible cuspidal $k$-representation of $\rM$ which contains $(J_{\rM},\lambda_{\rM})$, and $\pi'$ be an irreducible cuspidal $k$-representation of $\rM'$ which contains $(\tilde{J}_{\rM}',\tilde{\lambda}_{\rM}')$ such that $\pi' \hookrightarrow \pi\vert_{\rM'}$. We denote by 
$$\mathcal{P}_{[\rM,\pi]}=i_{\rM}^{\rG}\ind_{J_{\rM}}^{\rM}\mathcal{P}_{\lambda_{\rM}},$$
and by
$$\mathcal{P}_{[\rM',\pi']}=i_{\rM'}^{\rG'}\ind_{\tilde{J}_{\rM}'}^{\rM'}\mathcal{P}_{\tilde{\lambda}_{\rM}'}.$$

\begin{lem}
\label{lemma 001}
$\mathcal{P}_{[\rM',\pi']}$ is an direct summand of $\mathcal{P}_{[\rM,\pi]}\vert_{\rM'}$.
\end{lem}

\begin{proof}
We have 
$$(i_{\rM}^{\rG}\ind_{J_{\rM}}^{\rM}\mathcal{P}_{\lambda_{\rM}})\vert_{\rG'}\cong i_{\rM'}^{\rG'}(\ind_{J_{\rM}}^{\rM}\mathcal{P}_{\lambda_{\rM}})\vert_{\rM'}.$$
Since $\ind_{J_{\rM}}^{\tilde{J}_{\rM}}\mathcal{P}_{\lambda}\vert_{\tilde{J}_{\rM}'}$ is projective and has a surjection to $\tilde{\lambda}_{\rM}'$, it is deduced that $\mathcal{P}_{\tilde{\lambda}_{\rM}'}$ is a direct summand of $\ind_{J_{\rM}}^{\tilde{J}_{\rM}}\mathcal{P}_{\lambda}\vert_{\tilde{J}_{\rM}'}$. Hence $\mathcal{P}_{\tilde{\lambda}_{\rM}'}$ is a direct summand of $\cP_{\tilde{\lambda}_{\rM}}\vert_{\rM'}$ where $\cP_{\tilde{\lambda}_{\rM}}\cong\ind_{J_{\rM}}^{\tilde{J}_{\rM}}\mathcal{P}_{\lambda_{\rM}}$, and $\mathcal{P}_{[\rM',\pi']}$ is a direct summand of $\mathcal{P}_{[\rM,\pi]}\vert_{\rG'}$.
\end{proof}

Let $(J_{\rM},\lambda_{\rM})$ be a maximal simple supercuspidal $k$-type of $\rM$, and $(\tilde{J}_{\rM}',\tilde{\lambda}_{\rM}')$ be a maximal simple supercuspidal $k$-type of $\rM'$ defined from $(J_{\rM},\lambda_{\rM})$ as in Section \ref{section 1.0}. 

\begin{lem}
\label{lem 0012}
Let $\pi$ be an irreducible supercuspidal $k$-representation of $\rM$ which contains $(J_{\rM},\lambda_{\rM})$, and $\tau'$ be an irreducible subquotient of the projective cover $\mathcal{P}_{\tilde{\lambda}_{\rM}'}$ of $\tilde{\lambda}_{\rM}'$. Then $(\tilde{J}_{\rM}',\tau')$ is also a maximal simple supercuspidal $k$-type defined by $(J_{\rM},\lambda_{\rM})$, and there exists an irreducible direct component $\pi_0'$ of $\pi\vert_{\rM'}$ which contains $(\tilde{J}_{\rM}',\tau')$. In particular, when $\rM'=\rG'=\mathrm{SL}_n(F)$, if $\tau'$ is different from $\tilde{\lambda}_{\rM}'$, and suppose $\pi'$ is an irreducible direct component of $\pi\vert_{\rM'}$ containing $\tilde{\lambda}_{\rM}'$, then $\pi_0'$ is different from $\pi'.$
\end{lem}

\begin{proof}
Recall that $\mathcal{P}_{\lambda_{\rM}}$ is the projective $k$-cover of $\lambda_{\rM}$, as explained in Remark \ref{rem 0014}, its irreducible subquotients are isomorphic to $\lambda_{\rM}$. As in the proof of Lemma \ref{lemma 001}, we know that the projective representation $\mathcal{P}_{\tilde{\lambda}_{\rM}'}$ is an indecomposable direct component of $\ind_{J_{\rM}}^{\tilde{J}_{\rM}}\mathcal{P}_{\lambda_{\rM}}\vert_{\tilde{J}_{\rM}'}$. As in Section \ref{section 1.0}, the induced representation $\tilde{\lambda}_{\rM}:=\ind_{J_{\rM}}^{\tilde{J}_{\rM}}\lambda_{\rM}$ is irreducible. By the exactness of induction functor, we know that the irreducible subquotients of $\ind_{J_{\rM}}^{\tilde{J}_{\rM}}\mathcal{P}_{\lambda_{\rM}}$ are isomorphic to $\tilde{\lambda}_{\rM}$, which implies that an irreducible subquotient of $\mathcal{P}_{\tilde{\lambda}_{\rM}'}$ is isomorphic to an irreducible direct component of $\tilde{\lambda}_{\rM}\vert_{\tilde{J}_{\rM}'}$. Since $\pi$ contains $\tilde{\lambda}_{\rM}$ after restricting to $\tilde{J}_{\rM}$, by the Mackey's theory, any irreducible direct component of $\tilde{\lambda}_{\rM}\vert_{\tilde{J}_{\rM}'}$ must be contained in an irreducible direct component of $\pi\vert_{\rM'}$.

When $\rM'=\rG'=\mathrm{SL}_n(F)$, by Mackey's theory we have the inducetion $\ind_{\tilde{J}_{\rG}'}^{\rG'}\res_{\tilde{J}_{\rG}'}^{\tilde{J}_{\rG}}\tilde{\lambda}_{\rG}$ is a sub-representation of $\pi\vert_{\rG'}$, and each irreducible component of $\tilde{\lambda}_{\rG}\vert_{\tilde{J}_{\rG}'}$ is irreducibly induced to $\rG'$. The second statement is directly from the fact that $\pi\vert_{\rM'}$ is multiplicity-free as proved in Proposition 2.35 of \cite{C1}.
\end{proof}

\begin{rem}
\label{rem 0033}
When $\rM'$ is a proper Levi subgroup of $\rG'$, it is possible that two different maximal simple supercuspidal $k$-types $(\tilde{J}_{\rM}',\tilde{\lambda}_{\rM}')$ and $(\tilde{J}_{\rM}',\tau')$, which are defined from a same maximal simple supercuspidal $k$-type, are $\rM'$-conjugate to each other, which implies that they may be contained in a same irreducible supercuspidal $k$-representation of $\rM'$.
\end{rem}

\begin{lem}
\label{lemma 004}
\begin{enumerate}
\item Let $\alpha\in\tilde{J}_{\rM}$, then $\alpha(\cP_{\lambda_{\rM}})\cong\cP_{\alpha(\lambda_{\rM})}\cong\cP_{\lambda_{\rM}}\otimes\theta$, where $\theta\in\det(J_{\rM})^{\wedge}$ and $\alpha(\lambda_{\rM})\cong\lambda_{\rM}\otimes\theta$.
\item Let $(\tilde{J}_{\rM}',\tilde{\lambda}_1')$ and $(\tilde{J}_{\rM}',\tilde{\lambda}_2')$ be two different maximal simple $k$-types defined from $(J_{\rM},\lambda_{\rM})$. Let $\alpha\in\tilde{J}_{\rM}$ such that $\alpha(\tilde{\lambda}_1')\cong\tilde{\lambda}_2'$, then for the projective covers we have $\alpha(\cP_{\tilde{\lambda}_1'})\cong\cP_{\tilde{\lambda}_2'}$.
\end{enumerate}
\end{lem}

\begin{proof}
For the first part, there is a surjective morphism from $\cP_{\lambda_{\rM}}\otimes\theta$ to $\lambda_{\rM}\otimes\theta$ and is indecomposable. Moreover, the projectivity can be easily deduced directly from the definition. Since $\alpha(\cP_{\lambda_{\rM}})$ is the projective cover of $\alpha(\lambda_{\rM})$, we obtain the expected equality. The second part can be deduced in a similar way.
\end{proof}

\begin{prop}
\label{prop 008}
Recall that $\rG'$ is a Levi subgroup of $\mathrm{SL}_n(F)$. Let $\rho'$ be an irreducible $k$-representation of $\rG'$ and $(\rM',\pi')$ be a cuspidal pair of $\rG'$ inside the cuspidal support of $\rho'$, then there is a surjective morphism
$$\cP_{[\rM',\pi']}\rightarrow\rho'.$$
\end{prop}

\begin{proof}
Let $(\tilde{J}_{\rM}',\tilde{\lambda}_{\rM}')$ be a maximal simple $k$-type contained in $\pi'$, hence there is an injection $\tilde{\lambda}_{\rM}'\rightarrow \res_{\tilde{J}_{\rM}'}^{\rM'}\pi'$. By Frobenius reciprocity, it gives a surjection $\ind_{\tilde{J}_{\rM}'}^{\rM'}\cP_{\tilde{\lambda}_{\rM}'}\rightarrow\pi'$, which induces a surjection $\cP_{[\rM',\pi']}\rightarrow i_{\rM'}^{\rG'}\pi'$, hence a surjection from $\cP_{[\rM',\pi']}$ to $\rho'$ by \cite[\S II, 2.20]{V1}.
\end{proof}

\begin{cor}
\label{cor 0010}
Let $\mathcal{I}_{[\rM',\pi']}$ be the contragredient of $\cP_{[\rM',\pi^{'\vee}]}$, where $\pi^{'\vee}$ is the contragredient of $\pi'$. Suppose that the cuspidal support of $\tau'$ is $[\rM',\pi']$, then $\tau'$ is embedding to $\mathcal{I}_{[\rM',\pi']}$.
\end{cor}

\subsection{Category decomposition}    
\label{section 4}

Recall that in this section $\rG'$ is a Levi subgroup of $\mathrm{SL}_n(F)$ and $\rG$ is a Levi subgroup of $\mathrm{GL}_n(F)$ such that $\rG'=\rG\cap\mathrm{SL}_n(F)$. A decomposition of $\mathrm{Rep}_k(\rG')$ by its full sub-categories will be given in Theorem \ref{thm 009} according to the $\rG$-twist equivalent supercuspidal classes of $\rG'$ (see the paragraph below Proposition \ref{prop 007} for $\rG$-twist equivalent equivalence). This will not be a block decomposition in general, which means it does not always verify the last condition of Equation (\ref{0001}), however we will see in Section \ref{section 5} that it is not always possible to decompose $\mathrm{Rep}_k(\rG')$ according to the $\rG'$-inertially equivalent supercuspidal classes as for $\mathrm{Rep}_k(\rG)$ in Equation (\ref{0001}). 

Let $\mathcal{A}$  be a familly of $\rG$-inertially equivalent supercuspidal classes of $\rG$, and denote by $\mathrm{Rep}_k(\rG)_{\mathcal{A}}$ the union of blocks $\bigcup_{[\rM,\pi]\in\mathcal{A}}\mathrm{Rep}_k(\rG)_{[\rM,\pi]}$. Let $\mathcal{A}'$ be a family of $\rG'$-inertially equivalent supercuspidal classes of $\rG'$, verifying that $[\rM',\pi']\in\mathcal{A'}$ if and only if there exists $[\rM,\pi]\in\mathcal{A}$ such that $\rM'=\rM\cap\rG'$ and $\pi'\rightarrow\pi\vert_{\rM'}$. Let $\rL$ be a Levi subgroup of $\rG$ which contains $\rM$, denote by $\mathcal{A}_{\rL}$ the family of $\rL$-inertially equivalent supercuspidal classes in the form of $[w(\rM),w(\pi)]_{\rM}$, where $[\rM,\pi]\in\mathcal{A}$, and recall that $[\cdot,\cdot]_{\rL}$ is the $\rL$-inertially equivalent class, and $w\in\rG$ such that $w(\rM)\subset\rL$.

\begin{lem}
\label{lemma 002}
Let $P\in\mathrm{Rep}_k(\rG)_{[\rM,\pi]}$, and $\rL$ be a Levi subgroup of $\rG$. Then $r_{\rL}^{\rG}P\in\bigcup_{w\in\rG,w(\rM)\subset\rL}\mathrm{Rep}_k(\rG)_{[w(\rM),w(\pi)]_{\rL}}$.
\end{lem}

\begin{proof}
Suppose $\Pi$ is an irreducible sub-quotient of $r_{\rL}^{\rG}P$, whose cuspidal support is $(\rN,\tau)$, where $\rN$ is a Levi subgroup of $\rL$ and $\tau$ is a cuspidal representation of $\rN$. Let $\mathcal{P}_{[\rN,\tau]}$ be the projective object defined from the maximal simple $k$-type of $\tau$, then there is a non-trivial morphism $\mathcal{P}_{[\rN,\tau]}\rightarrow r_{\rL}^{\rG}P$. By the second adjonction of Bernstein, we have a non-trivial morphism from $\overline{i_{\rL}^{\rG}}\mathcal{P}_{[\rN,\tau]}$ to $P$, where $\overline{i_{\rL}^{\rG}}$ is the opposite normalised parabolic induction from $\rL$ to $\rG$. Since the module character $\delta_{\rL}$ is an unramified character on $\rL$, the $k$-representation $\overline{i_{\rL}^{\rG}}\mathcal{P}_{[\rN,\tau]}$ belongs to the same block as $i_{\rL}^{\rG}\mathcal{P}_{[\rN,\tau]}$, which implies that the supercuspidal support of $\tau$ belongs to the union $\cup_{w\in\rG,w(\rM)\subset\rL}(w(\rM),w(\pi))$. We finish the proof.
\end{proof}

\begin{lem}
\label{lemma 003}
Let $P\in\mathrm{Rep}_k(\rG)_{\mathcal{A}}$, and $\tau'$ be an irreducible subquotient of $P\vert_{\rG'}$, then the supercuspidal support of $\tau'$ belongs to $\mathcal{A}'$.
\end{lem}

\begin{proof}
Suppose firstly that $\tau'$ is cuspidal, then there exists a maximal simple $k$-type $(J,\lambda)$ of $\rG$, such that an irreducible component $(J',\lambda')$ of $\lambda\vert_{\rG'}$ is contained in $\tau'$ as a subrepresentation. By \cite[Lemma 2.14]{C1}, up to twist a $k$-character of $F^{\times}$, we can assume that $\lambda$ is a subquotient of $P\vert_J$. Hence there is a non-trivial morphism from the projective cover $\mathcal{P}_{\lambda}$ of $\lambda$ to $P\vert_J$, which   implies that for any irreducible cuspidal $k$-representation $\tau$ of $\rG$ which contains $(J,\lambda)$, its supercuspidal support must belong to $\mathcal{A}$. In particular, we can choose $\tau$ such that $\tau'\hookrightarrow \tau\vert_{\rG'}$, hence by \cite[Proposition 4.4]{C2} we know the supercuspidal support of $\tau'$ must belong to $\mathcal{A}'$.

Now suppose $\tau'$ is not cuspidal. Let $(\rL',\rho')$ belong to its cuspidal support. The $\rho'$ appears as a subquotient of $r_{\rL'}^{\rG'}P\vert_{\rG'}\cong (r_{\rL}^{\rG}P)\vert_{\rL'}$.  By Lemma \ref{lemma 002}, and the previous paragraph, we know that the supercuspidal support of $\rho'$ must belong to $\mathcal{A}'_{\rL}$, from which we deduces the desired property of supercuspidal support of $\tau$.
\end{proof}

\begin{lem}
\label{lemma 005}
Let $\pi$ and $\pi'$ be cuspidal $k$-representation of $\rM$ and $\rM'$ respectively and $\pi'\hookrightarrow\pi\vert_{\rM'}$. Let $(J_{\rM},\lambda_{\rM})$ be a maximal simple $k$-type of $\pi$ and $(\tilde{J}_{\rM}',\tilde{\lambda}_{\rM}')$ be a maximal simple $k$-type of $\pi'$ defined from $(J_{\rM},\lambda_{\rM})$. Suppose $[\rL,\tau]$ is the supercuspidal support of $[\rM,\pi]$, then we have
$$\ind_{\rG'}^{\rG}\mathcal{P}_{[\rM',\pi']}\in\prod_{\chi\in(\mathcal{O}_F^{\times})^{\vee}}\mathrm{Rep}_k(\rG)_{[\rL,\tau\otimes\chi]}.$$
\end{lem}

\begin{proof}
We denote by $\mathcal{P}'=\mathcal{P}_{[\rM',\pi']}$, $\cP_{\rM'}'=\ind_{\tilde{J}_{\rM'}}^{\rM'}\cP_{\tilde{\lambda}_{\rM}'}$ and $\cP_{\tilde{\lambda}_{\rM}}=\ind_{J_{\rM}}^{\tilde{J}_{\rM}}\cP_{\lambda_{\rM}}.$is this proof. Recall that $\mathcal{P}'=i_{\rM'}^{\rG'}\ind_{\tilde{J}_{\rM}'}^{\rM'}\tilde{\cP}_{\tilde{\lambda}_{\rM}'}$. Since the module character $\delta_{\rM'}=\delta_{\rM}\vert_{\rM'}$, we have
$$\ind_{\rG'}^{\rG}\cP'\cong i_{\rM}^{\rG}\ind_{\rM'}^{\rM}\cP_{\rM'}'\hookrightarrow i_{\rM}^{\rG}\ind_{\tilde{J}_{\rM}}^{\rM}(\cP_{\tilde{\lambda}_{\rM}}\otimes\ind_{\tilde{J}_{\rM}'}^{\tilde{J}_{\rM}}\mathds{1}),$$
where 
\begin{equation}
\label{equation 001}
\begin{aligned}
\res_{J_{\rM}^1}^{\tilde{J}_{\rM}}(\cP_{\tilde{\lambda}_{\rM}}\otimes\ind_{\tilde{J}_{\rM}'}^{\tilde{J}_{\rM}}\mathds{1})&=\res_{J_{\rM}^1}^{\tilde{J}_{\rM}}\ind_{J_{\rM}}^{\tilde{J}_{\rM}}\cP_{\lambda_{\rM}}\otimes\res_{J_{\rM}^1}^{\tilde{J}_{\rM}}\ind_{\tilde{J}_{\rM}'}^{\tilde{J}_{\rM}}\mathds{1}\\
&=\oplus_{\alpha\in \tilde{J}_{\rM}\backslash J_{\rM}}\res_{J_{\rM}^1}^{J_{\rM}}\alpha(\cP_{\lambda_{\rM}})\otimes\oplus_{\tilde{J}_{\rM}'\backslash\tilde{J}_{\rM}\slash J_{\rM}^1}\ind_{J_{\rM}^1\cap\tilde{J}_{\rM}'}^{J_{\rM}^1}\mathds{1}.
\end{aligned}
\end{equation}
Since $J_{\rM}^1$ is a pro-$p$ group, and by the definition of $\tilde{J}_{\rM}$ the above representation is semisimple whose direct components are in the form of $\eta\otimes\theta$, where $\eta$ is the Heisenberg representation of the simple character of $\lambda_{\rM}$, and $\theta\in(\det(J_{\rM}^1))^{\vee}$ which can be extend to an character of $(F^{\times})^{\wedge}$ and we fix one of such extension by denoting it as $\theta$ as well. Hence we have the decomposition
\begin{equation}
\label{equation 002}
\res_{J_{\rM}}^{\tilde{J}_{\rM}}(\cP_{\tilde{\lambda}_{\rM}}\otimes\ind_{\tilde{J}_{\rM}'}^{\tilde{J}_{\rM}}\mathds{1})\cong\oplus_{\theta\in(\det(H_{\rM}^1))^{\wedge}}P_{\theta},
\end{equation}
where $P_{\theta}$ is the $\eta\otimes\theta$-isogeny subrepresentation. Noticing that we require $\theta$ is non-trivial on $H_{\rM}^1$, because otherwise $\eta\cong\eta\otimes\theta$. By a similar computation as in Equation (\ref{equation 001}), we have
$$\res_{J_{\rM}}^{\tilde{J}_{\rM}}(\cP_{\tilde{\lambda}_{\rM}}\otimes\ind_{\tilde{J}_{\rM}'}^{\tilde{J}_{\rM}}\mathds{1})\cong\oplus_{\alpha\in\tilde{J}_{\rM}\slash J_{\rM}}\alpha(\cP_{\lambda_{\rM}})\otimes\oplus_{\rho\in(\det(J_{\rM}))^{\wedge}}\rho\otimes\ind_{J_{\rM,\ell'}}^{J_{\rM}}\mathds{1},$$
where $J_{\rM,\ell}$ is the subgroup of $J_{\rM}$ consisting with the elements whose determinant belong to the $\ell'$-part of $F^{\times}$. By Lemma \ref{lemma 004}, the right hand side of the above equation is isomorphic to $\oplus_{\rho\in(\det(J_{\rM}))^{\wedge}}(\cP_{\lambda_{\rM}}\otimes\rho)^{\tilde{q}}\otimes\ind_{J_{\rM,\ell'}}^{J_{\rM}}\mathds{1}$, where $\tilde{q}$ is the index $[\tilde{J}_{\rM}:J_{\rM}]$ and $(\cdot)^{\tilde{q}}$ is the $\tilde{q}$-multiple of $\cdot$.  Hence $P_{\theta}$ in Equation (\ref{equation 002}) is isomorphic to $\oplus_{\rho}(\cP_{\lambda_{\rM}}\otimes\rho)^{\tilde{q}}\otimes\ind_{J_{\rM,\ell'}}^{J_{\rM}}\mathds{1}$, where $\rho\in(\det{J}_{\rM})^{\wedge},\rho\vert_{H_{\rM}^1}=\theta$. Recall that $\lambda_{\rM}\cong\kappa\otimes\sigma$, where $\sigma$ is inflated from a supercuspidal $k$-representation of $J_{\rM}\slash J_{\rM}^1$, and $\cP_{\lambda}\cong\cP_{\sigma}\otimes\kappa$. Since an irreducible sub-quotient of $P_{\theta}$ is isomorphic to $\kappa\otimes\sigma_0\otimes\rho$, where $\sigma_0$ is inflated from $J_{\rM}\slash J_{\rM}^1$ and its supercuspidal support is the same as that of $\sigma$, and $\rho$ is an character as above. 
Now we fix an extension of $\theta$ to $J_{\rM}$ and denote it again by $\theta$. We have $P_{\theta}\cong \Pi_{\theta}\otimes\kappa\otimes\theta$, where $\Pi_{\theta}$ is inflated from $J_{\rM}\slash J_{\rM}^1$, and the supercuspidal support of each of its irreducible sub-quotient is the same as $\sigma\otimes\overline{\rho}$, where $\overline{\rho}\in\det(J_{\rM}\slash J_{\rM}^1)^{\wedge}$. By \cite[Theorem 9.6]{SeSt}, the induction $\ind_{J_{\rM}}^{\rM}P_{\theta}$ belongs to the sub-category $\prod_{\chi\in (\mathfrak{o}_{F}^{\times})^{\wedge}}\mathrm{Rep}_k(\rM)_{[\rL,\tau\otimes\chi]}$, hence $i_{\rM}^{\rG}P_{\theta}\in\prod_{\chi\in(\mathfrak{o}_{F}^{\times})^{\wedge}}\mathrm{Rep}_k(\rG)_{[\rL,\tau\otimes\chi]}$. Since $\ind_{\rG'}^{\rG}\cP'\cong\oplus_{\theta\in(\det(H_{\rM}^1))^{\wedge}}i_{\rM}^{\rG}\ind_{J_{\rM}}^{\rM}P_{\theta}$, we deduce the desired property.
\end{proof}

\begin{lem}
\label{lemma 006}
Let $\mathcal{A}$ be as above, and $P\in\mathrm{Rep}_k(\rG)_{\mathcal{A}}$. Then $P^{\vee}\in\mathrm{Rep}_k(\rG)_{\mathcal{A}^{\vee}}$, where $\mathcal{A}^{\vee}$ consists of the $\rG$-inertially equivalent supercuspidal classes $[\rM,\pi^{\vee}]$ such that $[\rM,\pi]\in\mathcal{A}$.
\end{lem}

\begin{proof}
Suppose there exists an irreducible sub-quotient $\pi$ of $P^{\vee}$. Denote by $[\rM_0,\tau_0]$ its supercuspidal support and by $[\rL_0,\pi_0]$ is cuspidal support. There is non-trivial morphism $\mathcal{P}_{[\rL_0,\pi_0]}\rightarrow P^{\vee}$, which implies a non-trivial morphism $P\rightarrow \cP_{[\rL_0,\pi_0]}^{\vee}$. Since $\cP_{[\rL_0,\pi_0]}^{\vee}$ belongs to the block $\mathrm{Rep}_k(\rG)_{[\rM_0,\tau_0^{\vee}]}$ by \cite[Corollary 11.7]{Helm}, we have $[\rM_0,\tau_0^{\vee}]\in\mathcal{A}$ by the Bernstein decomposition \cite[Theorem 11.8]{Helm}. 

\end{proof}

\begin{prop}
\label{prop 007}
We take the same notations as in Lemma \ref{lemma 005}. Let $\rho'$ be an irreducible subquotient of the contragredient $\mathcal{P}_{[\rM',\pi']}^{\vee}$, then the supercuspidal support of $\rho'$ is contained in union of $\rG$-conjugation classes of $[\rL',\tau'^{\vee}]$. In other words, let $\tau\vert_{\rL'}=\oplus_{i\in I}\tau_i'$, then the supercuspidal support of $\rho'$ is contained in $\bigcup_{i\in I}[\rL',\tau_i'^{\vee}]$.
\end{prop}

\begin{proof}
Let $\cP^{'}$ be $\cP_{[\rM',\pi']}$ in this proof. Since there is no non-trivial character on $\rG'$, we have $(\ind_{\rG'}^{\rG}\cP')^{\vee}\cong\mathrm{Ind}_{\rG'}^{\rG}\cP^{'\vee}$. By Lemma \ref{lemma 005} and Lemma \ref{lemma 006}, we have
$$(\Ind_{\rG'}^{\rG}\cP')^{\vee}\in\prod_{\chi\in(\mathfrak{o}_{F}^{\times})^{\vee}}\mathrm{Rep}_k(\rG)_{[\rL,\tau^{\vee}\otimes\chi]}.$$
By the surjective morphism $\res_{\rG'}^{\rG}\Ind_{\rG'}^{\rG}\cP^{'\vee}\rightarrow\cP^{'\vee}$ and Lemma \ref{lemma 003}, we conclude that the supercuspidal support of an arbitrary irreducible sub-quotient of $\cP^{'\vee}$ belongs to the $\rG$-conjugation of $[\rL',\tau']$.
\end{proof}

\begin{defn}
\label{defn 0035}
Let $(\rL_1,\tau_1)$ and $(\rL_2,\tau_2)$ be cuspidal pairs of $\rG$, we say they are $\rG$-\textbf{twist equivalent}, if there exists $g\in\rG$ such that $g(\rL_1)=\rL_2$ and $g(\tau_1)$ is isomorphic to $\tau_2$ up to a $k$-quasicharacter of $F^{\times}$, which is an equivalence relation and denote by $[\rL_1,\tau_1]^{tw}$ the $\rG$-twist equivalent class defined by $(\rL_1,\tau_1)$.
\end{defn}

Noting that in the above definition, we do not require the $k$-quasicharacter of $F^{\times}$ is unramified, which is different comparing to the relation of $\rG$-inertially equivalence. We define the depth of a $\rG$-twist equivalent class as the minimal depth among all the pairs inside this class. Denote by $\mathcal{C}_{[\rL,\tau]^{tw}}$ the set of $\rG$-twist equivalent cuspidal classes whose supercuspidal support belong to $[\rL,\tau]^{tw}$ up to an isomorphism, and denote by $\mathcal{C}_{\overline{[\rL,\tau]^{tw}}}$ the set of $\rG$-twist equivalent cuspidal classes whose supercuspidal support does not belong to $[\rL,\tau]^{tw}$ up to isomorphism and a twist of a $k$-quasicharacter of $F^{\times}$. It is worth noticing that

\begin{enumerate}
\item $\mathcal{C}_{[\rL,\tau]^{tw}}$ is a finite set;
\item fix a positive number $n\in\mathbb{N}$, there are only finitely many object in $\mathcal{C}_{\overline{[\rL,\tau]^{tw}}}$ whose depth is smaller than $n$.
\end{enumerate}
Define
$$\mathcal{I}_{(\rL,\tau)}=\bigoplus_{[\rM',\pi']\in\mathcal{C}_{[\rL,\tau]^{tw}}'}\cP_{[\rM',\tau^{'\vee}]}^{\vee},$$
$$\mathcal{I}_{\overline{(\rL,\tau)}}=\bigoplus_{[\rM',\pi']\in\mathcal{C}_{\overline{[\rL,\tau]^{tw}}}'}\cP_{[\rM',\tau^{'\vee}]}^{\vee},$$
where the relation between $\mathcal{C}_{[\rL,\tau]^{tw}}$ and $\mathcal{C}_{[\rL,\tau]^{tw}}'$ is as explained in the beginning of Section \ref{section 4}.

\begin{lem}
$\mathcal{I}_{\overline{(\rL,\tau)}}$ is injective.
\end{lem}

\begin{proof}
In fact $\mathcal{I}_{\overline{(\rL,\tau)}}$ is the smooth part of the contragredient $\prod_{[\rM',\pi']\in\mathcal{C}_{\overline{[\rL,\tau]^{tw}}}'}\cP_{[\rM',\tau^{'\vee}]}^{\ast}$ of $\oplus_{[\rM',\pi']\in\mathcal{C}_{\overline{[\rL,\tau]^{tw}}}'}\cP_{[\rM',\tau^{'\vee}]}$, where $\cP_{[\rM',\tau^{'\vee}]}^{\ast}$ is the contragredient (not necessarily smooth) of $\cP_{[\rM',\tau^{'\vee}]^{tw}}$. Fix an open compact subgroup $K'$ of $\rG'$, there exist finitely many $[\rM',\pi']\in\mathcal{C}_{\overline{[\rL,\tau]^{tw}}}'$ such that the $K$-invariant part of $\cP_{[\rM',\tau^{'\vee}]}$ is non-trivial, which implies the same property for the contragredient $\cP_{[\rM',\tau^{'\vee}]}^{\ast}$ by \cite[\S 4.15]{V1}. Hence an $K$-invariant non-trivial linear form $f$ in the smooth part of $\prod_{[\rM',\pi']\in\mathcal{C}_{\overline{[\rL,\tau]^{tw}}}'}\cP_{[\rM',\tau^{'\vee}]}^{\ast}$ must belong to $\oplus_{[\rM',\pi']\in\mathcal{C}_{\overline{[\rL,\tau]^{tw}}}'}\cP_{[\rM',\tau^{'\vee}]}^{\vee}$, which finishes the proof.
\end{proof}

\begin{thm}
\label{thm 009}
Let $\rG$ be a Levi subgroup of $\mathrm{GL}_n(F)$ and $\rG'$ be a Levi subgroup of $\mathrm{SL}_n(F)$, such that $\rG'=\rG\cap\mathrm{SL}_n(F)$. We have a category decomposition
$$\mathrm{Rep}_k(\rG')\cong\prod_{[\rL,\tau]^{tw}\in\mathcal{SC}_{\rG}^{tw}}\mathrm{Rep}_k(\rG')_{[\rL,\tau]^{tw}},$$ 
where
\begin{enumerate}
\item $\mathcal{SC}_{\rG}^{tw}$ is the set of $\rG$-twist equivalent supercuspidal classes in $\rG$; 
\item $\mathrm{Rep}_k(\rG')_{[\rL,\tau]^{tw}}$ is the full sub-category consisting with the objects whose irreducible sub-quotients have supercuspidal support belonging to $[\rL',\tau']_{\rG}$ and $\tau'$ is an irreducible direct component of $\tau\vert_{\rL'}$.
\end{enumerate}
In particular, for each object in $\mathrm{Rep}_k(\rG')_{[\rL,\tau]^{tw}}$, it has an injective resolution with direct sums of copies of $\mathcal{I}_{(\rL,\tau)}$.
\end{thm}

\begin{proof}
By the definition of $\mathcal{I}_{(\rL,\tau)}$, Corollary \ref{cor 0010} and Proposition \ref{prop 007}, each irreducible sub-quotient of $\mathcal{I}_{(\rL,\tau)}$ is a sub-representation of $\mathcal{I}_{(\rL,\tau)}$, and non of the irreducible sub-quotient of $\mathcal{I}_{\overline{(\rL,\tau)}}$ appears as a sub-quotient of $\mathcal{I}_{(\rL,\tau)}$. Furthermore, each irreducible $k$-representation is either a sub-representation of $\mathcal{I}_{(\rL,\tau)}$ or a sub-representation of $\mathcal{I}_{\overline{(\rL,\tau)}}$ by the unicity of cuspidal support as well as the unicity of supercuspidal support. Hence by Proposition \ref{prop 0011}, for any object $\Pi\in\mathrm{Rep}_k(\rG')$ and any $\rG$-twist equivalent supercuspidal class $[\rL,\tau]^{tw}$ of $\rG$, define $\Pi_{[\rL,\tau]^{tw}}$ be the largest sub-representation of $\Pi$ belonging to $\mathrm{Rep}_k(\rG')_{[rL,\tau]^{tw}}$, we have $\Pi\cong\oplus_{[\rL,\tau]\in\mathcal{SC}_{\rG}}\Pi_{[\rL,\tau]^{tw}}$, and by applying Proposition \ref{prop 0011} we know that there is no morphism between objects of sub-categories defined from different $\rG$-twist equivalent supercuspidal classes, hence we finish the proof.
\end{proof}

\begin{rem}
\label{rem 0031}
Let $(\rL,\tau)$ be a supercuspidal pair of $\rG$, and $\tau\vert_{\rL'}\cong\oplus_{j=1}^s\tau_j'$, where $(\rL',\tau_j')$ are supercuspidal pairs of $\rG'$. Denote by $\mathrm{Rep}_k(\rG')_{[\rL',\tau_j']}$ the full sub-category of $\mathrm{Rep}_k(\rG')$, consisting of objects of which any irreducible subquotient has supercuspidal support belonging to the $\rG'$-inertially equivalent class $[\rL',\tau_j']$. The sub-category $\mathrm{Rep}_k(\rG')_{[\rL,\tau]^{tw}}$ is generated by sub-categories $\mathrm{Rep}_k(\rG')_{[\rL',\tau_j']}$, for all $1\leq j\leq s$. In other words, let $\mathcal{SC}_{\rG'}^{\rG}$ be the set of $\rG$-inertially equivalent supercuspidal classes of $\rG'$, Theorem \ref{thm 009} gives a category decomposition of $\mathrm{Rep}_k(\rG')$ according to $\mathcal{SC}_{\rG'}^{\rG}$.
\end{rem}

\begin{cor}
\label{cor 0011}
Let $[\rL,\tau]$ be a $\rG$-inertially equivalent class of $\rG$, where $\rG$ is a Levi subgroup of $\mathrm{GL}_n(F)$. The functor $\res_{\rG'}^{\rG}$ gives functors from blocks $\mathrm{Rep}_k(\rG)_{[\rL,\tau\otimes\chi]}$ for any $k$-quasicharacter $\chi$ of $F^{\times}$ to the sub-category $\mathrm{Rep}_k(\rG')_{[\rL,\tau]^{tw}}$.
\end{cor}

\begin{proof}
It is directly from Theorem \ref{thm 009} and Lemma \ref{lemma 003}.
\end{proof}

\begin{cor}
\label{cor 0022}
Let $\rG'$ be a Levi subgroup of $\mathrm{SL}_n(F)$. There is a category decomposition
$$\mathrm{Rep}_k(\rG')\cong\mathrm{Rep}_k(\rG')_{\mathcal{SC}}\times\mathrm{Rep}_k(\rG')_{non-\mathcal{SC}},$$
where
\begin{enumerate}
\item an object belongs to $\mathrm{Rep}_k(\rG')_{\mathcal{SC}}$, if and only if its irreducible sub-quotients are supercuspidal;
\item an object belongs to $\mathrm{Rep}_k(\rG')_{non-\mathcal{SC}}$, if and only if non of its irreducible sub-quotients is supercuspidal.
\end{enumerate}
\end{cor}

\begin{proof}
Directly from Theorem \ref{thm 009}.
\end{proof}

\begin{defn}
\label{defn 0034}
We call $\mathrm{Rep}_k(\rM')_{\mathcal{SC}}$ the supercuspidal sub-categroy of $\mathrm{Rep}_k(\rM')$, and the blocks of $\mathrm{Rep}_k(\rM')_{\mathcal{SC}}$ are called supercuspidal blocks of $\mathrm{Rep}_k(\rM')$.
\end{defn}

\section{Supercuspidal sub-category of $\mathrm{Rep}_k(\rM')$}
\label{section 5}

In this section, let $\rG$ be $\mathrm{GL}_n(F)$ and $\rG'$ be $\mathrm{SL}_n(F)$. In the previous section, Theorem \ref{thm 009} gives a category decomposition of $\mathrm{Rep}_k(\rG')$, according to which we define the supercuspidal sub-category $\mathrm{Rep}_k(\rG')_{\mathcal{SC}}$. In this section, Theorem \ref{thm 0030} gives a description of the blocks of the supercuspidal sub-category of $\mathrm{Rep}_k(\rG')$ and $\mathrm{Rep}_k(\rM')$, where $\rM'$ is a Levi subgroup of $\rG'$.

\subsection{$\rM'$-inertially equivalent supercuspidal classes}

In this section, we give a bijection between $\rM'$-conjugacy classes of maximal simple $k$-types of $\rM'$, and $\rM'$-inertially equivalent cuspidal classes of $\rM'$. The most complexity of this section comes from the fact that the Levi subgroup of $\rG'$ is not a special linear group in lower rank.

Let $\rM$ be the Levi subgroup of $\rG$ such that $\rM'=\rM\cap\rG'$. Let $(\tilde{J}_{\rM}',\tilde{\lambda}_{\rM}')$ be a maximal simple $k$-type of $\rM'$ defined from a maximal simple $k$-type $(J_{\rM},\lambda_{\rM})$ of $\rM$. As explained in Section \ref{section 1.0}, let $\pi$ be an irreducible cuspidal $k$-representation of $\rM$ containing $(J_{\rM},\lambda_{\rM})$, then there exists a direct component $\pi'$ of $\pi\vert_{\rM'}$, such that $\pi'$ contains $(\tilde{J}_{\rM}',\tilde{\lambda}_{\rM}')$.

\begin{lem}
\label{lem 0020}
Let $E$ be a field extension of $F$, such that there is an embedding $E^{\times}\hookrightarrow \mathrm{GL}_n(F)$. Let $\omega_{E}$ be a uniformiser of $E$, and $Z_{\omega_{E}}$ be a subgroup of $\mathrm{GL}_n(F)$ generated by the image of $\omega_{E}$ under the embedding. Then a $k$-character of $Z_{\omega_{E}}$ can be extended to a character of $\mathrm{GL}_n(F)$.
\end{lem}

\begin{proof}
A $k$-character of $Z_{\omega_{E}}$ factors through determinant of $\mathrm{GL}_n(F)$. 
\end{proof}

Under the assumption on $E$ as in Lemma \ref{lem 0020}, denote by $Z_{\cO_E}$ the group generated by the image of $\cO_{E}^{\times}$ under the embedding. For general Levi subgroup $\rM$ of $\rG=\mathrm{GL}_n(F)$. Suppose $\rM$ is a direct product of $m$ general linear groups, and there exist field extensions $E_i,1\leq i\leq m$ of $F$, such that $\prod_{i=1}^mE_i^{\times}\hookrightarrow \rM$. Then after fixing a uniformiser $\omega_i$ for each $E_i$, we denote by $Z_{\omega_{E_{\rM}}}$ the group generated by the image of $\{1\times...\times\omega_i\times...\times1,1\leq i\leq m\}$ under the embedding, and by $Z_{\cO_{E_{\rM}}}$ the group generated by the image of $\prod_{i=1}^m\cO_{i}^{\times}$, where $\cO_{i}$ is the ring of integers of $E_i$. It is obvious that the image of $\prod_{i=1}^mE_i^{\times}$ can be decomposed as a direct product $Z_{\omega_{E_{\rM}}}\times Z_{\cO_{E_{\rM}}}$. In particular, when $E_i=F$ for $1\leq i\leq m$, we consider the canonical embedding, which is the equivalence between $(F^{\times})^m$ and the centre of $\rM$. Then the centre $Z_{\rM}$ of $\rM$ decomposes as $Z_{\omega_{F_{\rM}}}\times Z_{\cO_{F_{\rM}}}$. We denote by $Z_{\omega_{E_{\rM}}}'$ as $Z_{\omega_{E_{\rM}}}\cap\rM'$ and $Z_{\cO_{E_{\rM}}}'$ as $Z_{\cO_{E_{\rM}}}\cap\rM'$.

\begin{rem}
\label{remark 0021}
Lemma \ref{lem 0020} implies that a $k$-character of $Z_{\omega_{F_{\rM}}}$ can be extended to a $k$-character of $\rM$. In particular, for two irreducible $k$-representations of $\rM$, if their central characters coincide to each other on $Z_{\cO_{F_{\rM}}}$, then up to modifying by an unramified $k$-character, they share the same central character.
\end{rem}

\begin{prop}
\label{lem 0019}
Let $\pi_1,\pi_2$ be two irreducible cuspidal $k$-representations of $\rM'$ which contain $(\tilde{J}_{\rM}',\tilde{\lambda}_{\rM}')$. Then there exits an unramified $k$-character $\chi$ of $F^{\times}$, such that $\pi_1\cong\pi_2\otimes\chi$.
\end{prop}

\begin{proof}
Let $N_{\rM'}(\tilde{\lambda}_{\rM}')$ be the normaliser of $\tilde{\lambda}_{\rM}'$ in $\rM'$, which contains the centre $Z_{\rM'}$ of $\rM'$ as mentioned in Section \ref{section 1.0}, by Theorem 4.4 of \cite{C1} there exist extensions $\Lambda_{\rM',1},\Lambda_{\rM',2}$ of $\tilde{\lambda}_{\rM}'$ to $N_{\rM'}(\tilde{\lambda}_{\rM}')$, such that $\pi_1\cong\ind_{N_{\rM'}(\tilde{\lambda}_{\rM}')}^{\rM'}\Lambda_{\rM',1}$ and $\pi_2\cong\ind_{N_{\rM'}(\tilde{\lambda}_{\rM}')}^{\rM'}\Lambda_{\rM',2}$.

After modifying an unramified $k$-character of $\rM'$, we can assume that $\Lambda_{\rM',1}$ and $\Lambda_{\rM',2}$ have the same central character on $Z_{\rM'}$. 
In fact, we have $Z_{\cO_{F_{\rM}}}'\subset J_{\rM}'\subset\tilde{J}_{\rM}'$, hence the central characters of $\Lambda_{\rM',1}$ and $\Lambda_{\rM',2}$ coincide on $Z_{\cO_{F_{\rM}}}'$. On the other hand, since $Z_{\omega_{F_{\rM}}}\cong\mathbb{Z}^{m}$ for an integer $m$ decided by $\rM$, a character of a sub-$\mathbb{Z}$-module of $Z_{\omega_{F_{\rM}}}$ can be extended to $Z_{\omega_{F_{\rM}}}$. In particular, we can extend a character of  $Z_{\omega_{F_{\rM}}}'$ to $Z_{\omega_{F_{\rM}}}$, then to $\rM$ by Lemma \ref{lem 0020}, finally restricting to $\rM'$. Hence we prove that a character of $Z_{\omega_{F_{\rM}}}'$ can be extended to $\rM'$. Combining with the above discussion, we conclude that there is an unramified $k$-character $\chi_1$ of $\rM'$, such that $\Lambda_{\rM',1}\otimes\chi_1\vert_{Z_{\rM'}\tilde{J}_{\rM}'}\cong\Lambda_{\rM',2}\vert_{Z_{\rM'}\tilde{J}_{\rM}'}$. By the Frobenius reciprocity, there is an injection 
\begin{equation}
\label{equation 0021}
\Lambda_{\rM',1}\otimes\chi_1\hookrightarrow\Lambda_{\rM',2}\otimes\ind_{Z_{\rM'}\tilde{J}_{\rM}'}^{N_{\rM'}(\tilde{\lambda}_{\rM}')}\mathds{1}.
\end{equation}
As in Remark 2.42 of \cite{C1}, the group $N_{\rM'}(\tilde{\lambda}_{\rM}')$ (see Section \ref{section 1.0} for definition) is a subgroup with finite index of $E_{\rM}^{\times}\tilde{J}_{\rM}\cap\rM'$, where $E_{\rM}^{\times}\cong\prod_{i=1}^m E_i^{\times}$ and $E_{i}$ is a field extension of $F$ for each $1\leq i\leq m$.
Since the quotient group $N_{\rM'}(\tilde{\lambda}_{\rM}')\slash Z_{\rM}'\tilde{J}_{\rM}'$ is isomorphic to a subquotient group of $Z_{\omega_{E_{\rM}}}$, hence a character of $N_{\rM'}(\tilde{\lambda}_{\rM}')\slash Z_{\rM'}\tilde{J}_{\rM}'$ can be extended to a character of $\rM$ by Lemma \ref{lem 0020}, hence a character of $\rM'$. 

Now we look back to Equation \ref{equation 0021}. The $k$-representation $\ind_{Z_{\rM'}\tilde{J}_{\rM}'}^{N_{\rM'}(\tilde{\lambda}_{\rM}')}\mathds{1}$ has finite length and each of its irreducible subquotient is a character of $N_{\rM'}(\tilde{\lambda}_{\rM}')\slash Z_{\rM}'\tilde{J}_{\rM}'$, hence can be viewed as a character of $\rM'$. By the unicity of Jordan-H$\ddot{\mathrm{o}}$lder factors, there exists a character $\chi_2$ of $\rM'$, such that $\Lambda_{\rM',1}\otimes\chi_1\cong\Lambda_{\rM',2}\otimes\chi_2$, since $\chi_1,\chi_2$ are $k$-characters of $\rM'$, applying the induction functor $\ind_{N_{\rM'}(\tilde{\lambda}_{\rM}')}^{\rM'}$ on both sides gives an equivalence that $\pi_1\otimes\chi_1\cong\pi_2\otimes\chi_2$. Define $\chi$ to be $\chi_2\chi_1^{-1}$, which is the required unramified $k$-character of $\rM'$.
\end{proof}

\begin{prop}
\label{rem 0016}
Let $(\tilde{J}_{\rM}',\tilde{\lambda}_{\rM}')$ be a maximal simple $k$-type of $\rM'$, and $\pi'$ an irreducible $k$-representation of $\rM'$ containing $(\tilde{J}_{\rM}',\tilde{\lambda}_{\rM}')$. Then any irreducible subquotient of $\ind_{\tilde{J}_{\rM}'}^{\rM'}\tilde{\lambda}_{\rM}'$ must belong to $[\rM',\pi']_{\rM'}$, or equivalently saying, must be $\rM'$-inertially equivalent to $\pi'$.
\end{prop}

\begin{proof}
By Proposition $\mathrm{IV.}$1.6 of \cite{V2}, we know that $\ind_{J_{\rM}}^{\rM}\lambda_{\rM}$ is cuspidal, hence its sub-representation $\ind_{\tilde{J}_{\rM}'}^{\rM'}\tilde{\lambda}_{\rM}'$ is cuspidal as well. Then an irreducible subquotient $\pi_0$ is cuspidal, which contains a maximal simple $k$-type $(J_0',\lambda_0')$, which is weakly intertwined with $(\tilde{J}_{\rM}',\tilde{\lambda}_{\rM}')$ in $\rM'$ by Mackey's theory. By the property of weakly intertwining implying conjugacy of maximal simple $k$-types of $\rM'$ in Theorem 3.25 of \cite{C1}, we conclude that a maximal simple $k$-type contained in $\pi_0$ must $\rM'$-conjugate to $(\tilde{J}_{\rM}',\tilde{\lambda}_{\rM}')$, and hence $\pi_0$ contains $(\tilde{J}_{\rM}',\tilde{\lambda}_{\rM}')$. By Proposition \ref{lem 0019}, we conclude that $\pi_0$ is $\rM'$-inertially equivalent to $\pi'$.
\end{proof}

\begin{rem}
\label{rem 0023}
\begin{enumerate}
\item Lemma \ref{lem 0019} and Proposition $\ref{rem 0016}$ give a bijection between the set of $\rM'$-conjugacy classes of maximal simple $k$-types and the set of $\rM'$-inertially equivalent cuspidal classes:
$$\nu: [\tilde{J}_{\rM}',\tilde{\lambda}_{\rM}']_{\rM'}\mapsto[\rM',\pi']_{\rM'},$$
where $[\tilde{J}_{\rM}',\tilde{\lambda}_{\rM}']_{\rM'}$ is the $\rM'$-conjugacy class of $(\tilde{J}_{\rM}',\tilde{\lambda}_{\rM}')$, and $\pi'$ is an irreducible cuspidal $k$-representation contains $(\tilde{J}_{\rM}',\tilde{\lambda}_{\rM}')$.
\item Let $(\tilde{J}_{\rM}',\tilde{\lambda}_1')$ and $(\tilde{J}_{\rM}',\tilde{\lambda}_2')$ be two different maximal simple $k$-types defined by $(J_{\rM},\lambda_{\rM})$. When $\rM'=\rG'=\mathrm{SL}_n(F)$ by Lemma \ref{lem 0012}, the associated $\rG'$-inertially equivalent cuspidal classes defined of $(\tilde{J}_{\rM'},\tilde{\lambda}_i),i=1,2$ are differen. When $\rM'$ is a proper Levi of $\mathrm{SL}_n(F)$ by Remark \ref{rem 0033} , the associated $\rG'$-inertially equivalent cuspidal classes may be the same.
\end{enumerate}
\end{rem}

\subsection{Supercuspidal blocks of $\mathrm{Rep}_k(\rM')$}
\label{section 0028}
In this Section, we give a block decomposition of the supercuspidal sub-category $\mathrm{Rep}_k(\rM')_{\mathcal{SC}}$ of $\mathrm{Rep}_k(\rM')$, of which the blocks are called supercuspidal blocks of $\mathrm{Rep}_k(\rM')$ as defined in the end of Section \ref{section 4}. Let $[\rM',\pi']_{\rM'}$ be a $\rM'$-inertially equivalent supercuspidal class of $\rM'$, denote by $\mathrm{Rep}_k(\rM')_{[\rM',\pi']}$ the full sub-category of $\mathrm{Rep}_k(\rM')$, such that the irreducible sub-quotients of an object of $\mathrm{Rep}_k(\rM')_{[\rM',\pi']}$ belong to $[\rM',\pi']_{\rM'}$. As in Proposition of \cite{V2}[\S $\mathrm{III}$], a sub-category $\mathrm{Rep}_k(\rM')_{[\rM',\pi']}$ is non-split, and a block of $\mathrm{Rep}_k(\rM')_{\mathcal{SC}}$ is generated by a finitely number of subcategories in the form of $\mathrm{Rep}_k(\rM')_{[\rM',\pi']}$.

Let $(J_{\rM},\lambda_{\rM})$ be a maximal simple supercuspidal $k$-type of $\rM$, and $(\tilde{J}_{\rM}',\tilde{\lambda}_{\rM}')$ be a maximal simple supercuspidal $k$-type defined from $(J_{\rM},\lambda_{\rM})$ as explained in Section \ref{section 1.0}. Recall that $\cP_{\tilde{\lambda}_{\rM}'}$ is the projective cover of $\tilde{\lambda}_{\rM}'$. By Lemma \ref{lem 0012}, its irreducible subquotients are maximal simple supercuspidal $k$-types of $\rM'$ as well, and we denote by $\mathcal{I}(\tilde{\lambda}_{\rM}')$ the set of isomorphic classes of irreducible subquotients of $\cP_{\tilde{\lambda}_{\rM}'}$. We define a set of $\rM'$-inertially equivalent supercuspidal classes $\mathcal{SC}(\tilde{\lambda}_{\rM}')$, such that there is a bijection
$$\nu:\cI(\tilde{\lambda}_{\rM}')\rightarrow \mathcal{SC}(\tilde{\lambda}_{\rM}'),$$
which is given as in Remark \ref{rem 0023}.

\begin{prop}
\label{prop 0017}
Suppose the image $\mathcal{SC}(\tilde{\lambda}_{\rM}')$ is not a singleton. For any non-trivial disjoint union $\mathcal{SC}(\tilde{\lambda}_{\rM}')=\mathcal{SC}_1\sqcup\mathcal{SC}_2$, and let $\cI(\tilde{\lambda}_{\rM}')=\cI_1\sqcup\cI_2$ such that $\mathcal{SC}_1=\nu(\cI_1)$ and $\mathcal{SC}_2=\nu(\cI_2)$, it is not possible to decompose $\ind_{\tilde{J}_{\rM}'}^{\rM'}\cP_{\tilde{\lambda}_{\rM}'}$ as $P_1\oplus P_2$, where any irreducible subquotients of $P_1$ belongs to $\mathcal{SC}_1$ and any irreducible subquotients of $P_2$ belongs to $\mathcal{SC}_2$.
\end{prop}

\begin{proof}
We abbreviate $\ind_{\tilde{J}_{\rM}'}^{\rM'}\mathcal{P}_{\tilde{\lambda}_{\rM}'}$ by $\cP_{\rM'}$ in this proof. By Theorem \ref{thm 009}, the irreducible subquotients of $\cP_{\rM'}$ are supercuspidal. Suppose the contrary that, there exists a non-trivial disjoint union $\mathcal{SC}(\tilde{\lambda}_{\rM}')=\mathcal{SC}_1\sqcup\mathcal{SC}_2$, such that $\cP_{\rM'}=P_1\oplus P_2$ verifying the conditions in the statement of the propersition. Without loss of generality, we suppose $\tilde{\lambda}_{\rM}'\in\cI_1$. Let $\iota_{\rM}'$ be a maximal simple supercuspidal $k$-type in $\cI_2$, and $\tau'$ be a supercuspidal $k$-representation of $\rM'$ containing $\iota_{\rM}'$. Hence $\tau'$ is a sub-representation of $\ind_{\tilde{J}_{\rM}'}^{\rM'}\iota_{\rM'}$, and the later is a subquotient of $\cP_{\rM'}$, hence $P_2$ is non-trivial ($P_1$ is also non-trivial since $\tilde{\lambda}_{\rM}'\in\cI_1$).


By Lemma \ref{lem 0012}, there exists a filtration of $\{0\}=W_0\subset W_1...\subset W_{s}=\mathcal{P}_{\tilde{\lambda}_{\rM}'}$ for an $s\in\mathbb{N}$, such that each quotient $\tilde{\lambda}_i':=W_i\slash W_{i-1},1\leq i\leq s$ is irreducible and $(\tilde{J}_{\rM}',\tilde{\lambda}_i')$ is a maximal simple supercuspidal $k$-type of $\rM'$ defined also from $(J_{\rM},\lambda_{\rM})$. In particular, $\tilde{\lambda}_{\rM}'$ as well as $\iota_{\rM}'$ are isomorphic to $\tilde{\lambda}_i'$ for some $0\leq i\leq s$ respectively. Now define $\tilde{\lambda}_0'$ to be null, and denote by $V_i=\ind_{\tilde{J}_{\rM}'}^{\rM'} W_i$, then $\{ V_i\}_{0\leq i\leq s}$ is a filtration of $\cP_{\rM'}$ and $V_i\slash V_{i-1}\cong\ind_{\tilde{J}_{\rM}'}^{\rM'}\tilde{\lambda}_i', 1\leq i\leq s$. Denote by $V_{i,1}$ the image of $V_i$ in $P_1$ under the canonical projection, and $V_{i,2}$ the image of $V_i$ in $P_2$ under the canonical projection. Hence $\{ V_{i,1} \}_{0\leq i\leq s}$ (resp. $\{V_{i,2}\}_{0\leq i\leq s}$) forms a filtration of $P_1$ (resp. $P_2$). By Proposition \ref{rem 0016}, the quotient $V_{i,1}\slash V_{i-1,1}$ (resp. $V_{i,2}\slash V_{i-1,2}$) is non-trivial if and only if $\tilde{\lambda}_i'\in\cI_1$ (resp. $\tilde{\lambda}_i'\in\cI_2$).

Now we consider the canonical injective morphism 
$$\alpha:\cP_{\tilde{\lambda}_{\rM}'}\hookrightarrow\res_{\tilde{J}_{\rM}'}^{\rM'}\cP_{\rM'}.$$
Under the above assumption, we have $\res_{\tilde{J}_{\rM}'}^{\rM'}\cP_{\rM'}\cong\res_{\tilde{J}_{\rM}'}^{\rM'}P_1\oplus\res_{\tilde{J}_{\rM}'}^{\rM'}P_2$. Since we consider a representation of infinite length, the unicity of Jordan-H$\ddot{\mathrm{o}}$lder factors is not sufficient, and we need a simple but practical lemma as below to continue the proof:
\end{proof}

\begin{lem}
\label{lem 0018}
Let $\rG$ be a locally pro-finite group, and $\pi$ a $k$-representation of $\rG$. Let $\pi_1$ be a sub-representation of $\pi$. Suppose $\tau$ is an irreducible subquotient of $\pi$, then $\tau$ is either isomorphic to an irreducible subquotient of $\pi_1$ or to an irreducible subquotient of $\pi\slash\pi_1$. 
\end{lem}

\begin{proof}
Easy to check.
\end{proof}

\begin{proof}[Continue the proof of Proposition \ref{prop 0017}]
Suppose $\alpha(\cP_{\tilde{\lambda}_{\rM}'})\subset\res_{\tilde{J}_{\rM}'}^{\rM'}P_1$. Let $\iota_{\rM}'\in\cI_2$ be an irreducible subquotient of $\mathcal{P}_{\tilde{\lambda}_{\rM}'}$, by Lemma \ref{lem 0018} there exists $1\leq i\leq s$, such that $\iota_{\rM}'$ is an irreducible subquotient of $V_{i,1}\slash V_{i-1,1}$, and the later is a subquotient of $\ind_{\tilde{J}_{\rM}'}^{\rM'}\tilde{\lambda}_i'$. In other words, $\iota_{\rM}'$ is an irreducible subquotient of $\ind_{\tilde{J}_{\rM}'}^{\rM'}\tilde{\lambda}_i'$. Applying Mackey's theorem, it is equivalent to say that $\iota_{\rM}'$ is weakly intertwined with $\tilde{\lambda}_i'$ in $\rM'$ (see Section \ref{section 1.0}  for weakly intertwining), hence by Theorem 3.25 of \cite{C1} they are $\rM'$-conjugate to each other, hence they define the same $\rM'$-inertially equivalent class as in Remark \ref{rem 0023}. Meanwhile, by the above analysis, we know that $\nu(\tilde{\lambda}_i)\in\mathcal{SC}_1$ and $\nu(\iota_{\rM}')\in\mathcal{SC}_2$, which is a contradition. Hence $\alpha(\cP_{\tilde{\lambda}_{\rM}'})\cap \res_{\tilde{J}_{\rM}'}^{\rM'}P_1\neq \alpha(\cP_{\tilde{\lambda}_{\rM}'})$.

Now we consider $\alpha(\cP_{\tilde{\lambda}_{\rM}'})\slash (\alpha(\cP_{\tilde{\lambda}_{\rM}'})\cap \res_{\tilde{J}_{\rM}'}^{\rM'}P_1)$, which is non-null as above, and is a sub-representation of $\res_{\tilde{J}_{\rM}'}^{\rM'}\cP_{\rM'}\slash\res_{\tilde{J}_{\rM}'}^{\rM'}P_1\cong\res_{\tilde{J}_{\rM}'}^{\rM'}P_2$. By the same manner as above, we conclude that each irreducible subquotient of $\alpha(\cP_{\tilde{\lambda}_{\rM}'})\slash (\alpha(\cP_{\tilde{\lambda}_{\rM}'})\cap \res_{\tilde{J}_{\rM}'}^{\rM'}P_1)$ belongs to $\cI_2$, which implies that there exists $\tilde{\lambda}_{i_0}'\in\cI_2$ such that $\cP_{\tilde{\lambda}_{\rM}'}\rightarrow\tilde{\lambda}_{i_0}'$. Since $\tilde{\lambda}_{i_0}'$ is different from $\tilde{\lambda}_{\rM}'$, the maximal semisimple quotient of $\cP_{\tilde{\lambda}_{\rM}'}$ contains $\tilde{\lambda}_{i_0}'\oplus\tilde{\lambda}_{\rM}'$, which contradicts to the fact that $\cP_{\tilde{\lambda}_{\rM}'}$ is the projective cover of $\tilde{\lambda}_{\rM}'$ by Proposition 41 $c)$ \cite{Serre}. Hence we finish the proof.
\end{proof}

\begin{lem}
\label{lem 0024}
Let $(\tilde{J}_{\rM}',\tilde{\lambda}_1')$ and $(\tilde{J}_{\rM}',\tilde{\lambda}_2')$ be two maximal simple supercuspidal $k$-types. Suppose $\tilde{\lambda}_2'\in\cI(\tilde{\lambda}_1')$, then $\tilde{\lambda}_1'\in\cI(\tilde{\lambda}_2')$ (see the beginning of this section for the definition of $\cI(\cdot)$).
\end{lem}

\begin{proof}
Let $W(k)$ be the ring of Witt vectors of $k$, and $\cK$ be the fractional field of $W(k)$. Let $\tilde{\cK}$ be a finite field extension of $\cK$, such that $\tilde{\cK}$ contains the $\vert \tilde{J}_{\rM}\slash N\vert$-th roots, where $N$ is the kernel of $\cP_{\tilde{\lambda}_{\rM}}$, and let $\tilde{\cO}$ be its ring of integers. Consider the $\ell$-modular system $(\tilde{\cK},\tilde{\cO},k)$, we have that  $\cP_{\tilde{\lambda}_{\rM}}\otimes_{\tilde{\cO}}\tilde{\cK}$ is semisimple, whose direct components are absolutely irreducible. By Proposition 42 of \cite{Serre}, the projective cover $\cP_{\tilde{\lambda}_1'}$ can be lifted over $\tilde{\cO}$, and we denote the lifting to $\tilde{\cO}$ by $\cP_{\tilde{\lambda}_{\rM}}$ as well. Now we consider $\cP_{\tilde{\lambda}_1'}\otimes_{\tilde{\cO}}\tilde{\cK}$, which is semisimple with finite length. Suppose $P$ is an irreducible component of $\cP_{\tilde{\lambda}_1'}\otimes_{\tilde{\cO}}\tilde{\cK}$, then the semisimplification of its reduction modulo $\ell$ must contain $\tilde{\lambda}_1'$, otherwise it will induce a surjection from $\cP_{\tilde{\lambda}_1'}$ to an irreducible $k$-representation different from $\tilde{\lambda}_1'$, which contradicts with the fact the $\cP_{\tilde{\lambda}_1'}$ is the projective cover of $\tilde{\lambda}_1'$ by Proposition 41 of \cite{Serre}. Since $\tilde{\lambda}_2'$ is a subquotient of $\cP_{\tilde{\lambda}_1'}$, their exists an irreducible direct component $P_2'$ of $\cP_{\tilde{\lambda}_1'}\otimes_{\tilde{\cO}}\tilde{\cK}$, of which the semisimplification of reduction modulo-$\ell$ contains $\tilde{\lambda}_1'$ as well as $\tilde{\lambda}_2'$. Let $\alpha\in\tilde{J}_{\rM}$, such that $\alpha(\tilde{\lambda}_1')\cong\tilde{\lambda}_2'$. By the second part of Lemma \ref{lemma 004}, we have $\alpha(\cP_{\tilde{\lambda}_1'})\cong\cP_{\tilde{\lambda}_2'}$, which implies that $\alpha(P_2')$ is a direct component of $\cP_{\tilde{\lambda}_2'}$. We state that $\alpha$ stabilises $P_2'$. In fact, by the proof of Lemma \ref{lemma 001}, we have $\cP_{\tilde{\lambda}_1'}$ is an indecomposable direct factor of $\cP_{\tilde{\lambda}_{\rM}}$. In particular, the reduction modulo-$\ell$ of each irreducible components of $\cP_{\tilde{\lambda}_{\rM}}\otimes_{\tilde{\cO}}\tilde{\cK}$ is isomorphic to $\tilde{\lambda}_{\rM}$. By the uncity of Jordan-Holdar factors, there exists an irreducible component $P_2$ of $\cP_{\tilde{\lambda}_{\rM}}\otimes_{\tilde{\cO}}\tilde{\cK}$, such that $P_2'$ is an irreducible component of $P_2\vert_{\tilde{J}_{\rM}'}$. Since $\alpha(\tilde{\lambda}_1')\cong\tilde{\lambda}_2'$, the semisimplification of the reduction modulo-$\ell$ of $\alpha(P_2')$ contains $\tilde{\lambda}_2'$. Since $\alpha(P_2')$ is isomorphic to an irreducible component of $P_2\vert_{\tilde{J}_{\rM}'}$, and the reduction modulo $\ell$ of $P_2$ is isomorphic to $\tilde{\lambda}_{\rM}$, combining with the fact that $\tilde{\lambda}_{\rM}\vert_{\tilde{J}_{\rM}'}$ is multiplicity-free, we conclude that $\alpha(P_2')\cong P_2'$. Hence $\tilde{\lambda}_1'\in\cI(\tilde{\lambda}_2')$.
\end{proof}

\begin{defn}
\label{defn 0025}
Let $(J_{\rM},\lambda_{\rM})$ be a maximal simple supercuspidal $k$-type of $\rM$, and denote by $\cI(\lambda_{\rM})$ the set of isomorphic classes of maximal simple supercuspidal $k$-types of $\rM'$ defined by $(J_{\rM},\lambda_{\rM})$. Suppose $(\tilde{J}_{\rM}',\gamma')$ and $(\tilde{J}_{\rM}',\tau')$ be two elements in $\cI(\lambda_{\rM})$, we say
\begin{enumerate}
\item  $\gamma'$ is related to $\tau'$, if $\gamma'\in\cI(\tau')$ (or equivalently $\tau'\in\cI(\gamma')$ by Lemma \ref{lem 0024}) and we denote by $\gamma'\leftrightarrow\tau'$;
\item $\gamma'\sim\tau'$ if there exists a serie $(\tilde{J}_{\rM}',\tilde{\lambda}_i'), 1\leq i\leq t$ for an integer $t$, such that  
$$\gamma'\leftrightarrow\tilde{\lambda}_1'\leftrightarrow...\leftrightarrow\tilde{\lambda}_t'\leftrightarrow\tau',$$
and we call the series $\{\tilde{\lambda}_i',1\leq i\leq t\}$ a connected relation of $\gamma'$ and $\tau'$. The relation $``\sim"$ defines an equivalence relation on $\cI(\lambda_{\rM})$ (By Proposition 2.6 of \cite{C1}, the relation $\leftrightarrow$ and $\sim$ on $\mathcal{I}(\lambda_{\rM})$ does not depend on the choice of $\lambda_{\rM}$).
\item Denote by $[\tilde{\lambda}_{\rM}',\sim]$ the subset of $\cI(\lambda_{\rM})$ consisting of all $\tau'$ such that $\tau'\sim\tilde{\lambda}_{\rM}'$, or equivalently the connected component containing $\tilde{\lambda}_{\rM}'$ defined by $\sim$. 
\end{enumerate}
\end{defn}

Let $\pi$ be an irreducible supercuspidal $k$-representation of $\rM$, denote by $\cI(\pi)$ the isomorphic classes of the irreducible direct components of $\pi\vert_{\rM'}$. Let $(J_{\rM},\lambda_{\rM})$ be a maximal simple supercuspidal $k$-type contained in $\pi$. The above equivalence relation $``\sim"$ on $\cI(\lambda_{\rM})$ induces an equivalence relation on $\cI(\pi)$.

\begin{defn}
\label{defn 0026}
Let $\pi_1',\pi_2'\in\cI(\pi)$, we say $\pi_1'\sim\pi_2'$ if there exists a maximal simple supercuspidal $k$-type $(J_{\rM},\lambda_{\rM})$ contained in $\pi$, and two maximal simple supercuspidal $k$-types $(\tilde{J}_{\rM},\tilde{\lambda}_{\rM,1}')$ and $(\tilde{J}_{\rM},\tilde{\lambda}_{\rM,2}')$ defined from $(J_{\rM},\lambda_{\rM})$, such that $\pi_i'$ contains $\tilde{\lambda}_{\rM,i}'$ for $i=1,2$, and $\tilde{\lambda}_{\rM,1}'\sim\tilde{\lambda}_{\rM,2}'$. By the unicity property of intertwining implying conjugacy in Theorem 3.25 of \cite{C1}, $``\sim"$ defines an equivalence relation on $\cI(\pi)$.
\end{defn}

\begin{rem}
\label{rem 0027} Let $\pi'\in\cI(\pi)$, define $[\pi',\sim]$ to be a subset of $\cI(\pi)$, consisting of the elements that are equivalent to $\pi'$. In other words, $(\pi',\sim)$ is the connected component containing $\pi'$ under the equivalence relation $``\sim"$ on $\cI(\pi)$. In particular, there exists a subset $\{\pi_j',1\leq j\leq s\}$ of $\cI(\pi)$ for an integer $s$, such that $(\pi_j',\sim)$ are two-two disjoint, and $\cup_{j=1}^s(\pi_j',\sim)=\cI(\pi)$. Denote by $[\pi_j',\sim]$ the family of $\rM'$-inertially equivalent classes of $\pi'\in(\pi_j',\sim)$, and we call \textbf{$[\pi_j',\sim]$ a connected $\rM'$-inertially equivalent class of $\pi_j'$}.
\end{rem}


By Theorem \ref{thm 009}, to give a block decomposition of $\mathrm{Rep}_k(\rM')_{\mathcal{SC}}$ is equivalent to give a block decomposition of $\mathrm{Rep}_k(\rM')_{[\rM,\pi]^{tw}}$ for each irreducible supercuspidal $k$-representation $\pi$ of $\rM$. 
\begin{thm}[Block decomposition of $\mathrm{Rep}_k(\rM')_{\mathcal{SC}}$]
\label{thm 0030}
Let $\pi$ be an irreducible supercuspidal $k$-representation of $\rM$, and we occupy the notations in Remark \ref{rem 0027}. For each $1\leq j\leq s$, define the full sub-category $\mathrm{Rep}_k(\rM')_{[\pi_j',\sim]}$, consisting of the objects, of which each irreducible subquotient belongs to $[\pi_j',\sim]$. Then $\mathrm{Rep}_k(\rM')_{[\pi_j',\sim]}$ is a block, and the subcategroy $\mathrm{Rep}_k(\rM')_{[\rM,\pi]^{tw}}\cong\prod_{j=1}^s\mathrm{Rep}_k(\rM')_{[\pi_j',\sim]}$.
\end{thm}

\begin{proof}
First we prove that $\mathrm{Rep}_k(\rM')_{[\pi_j',\sim]}$ is non-split. By Proposition of \cite{V2}[\S $\mathrm{III}$], we only need to prove that for any non-trivial disjoint union $[\pi_j',\sim]=I_1\sqcup I_2$, where $I_1,I_2$ are two non-trivial families of $\rM'$-inertially equivalence classes, then there exists an object $P\in\mathrm{Rep}_k(\rM')_{[\pi_j',\sim]}$, such that $P$ cannot be decomposed as $P_1\oplus P_2$, where $P_1\in\mathrm{Rep}_k(\rM')_{I_1}$ and $P_2\in\mathrm{Rep}_k(\rM')_{I_2}$. Without loss of generality, we assume that $\pi_j'\in I_1$ and let $\pi_{j_0}'\in\cI(\pi)$ such that $\pi_{j_0}'\in I_2$. Since $\pi_j'\sim\pi_{j_0}'$, there exists a maximal simple supercuspidal $k$-type $(J_{\rM},\lambda_{\rM})$ of $\pi$ and two maximal simple supercuspidal $k$-types $(\tilde{J}_{\rM}',\tilde{\lambda}_{\rM}')$ of $\pi_j'$ and $(\tilde{J}_{\rM}',\tau_{\rM}')$ of $\pi_{j_0}'$, such that $\tilde{\lambda}_{\rM}'\sim\tau_{\rM}'$ in $\cI(\lambda_{\rM})$. By the second part of Definition \ref{defn 0025}, let $\{\tilde{\lambda}_i',1\leq i\leq t\}$ be a series of a connected relation of $\tilde{\lambda}_{\rM}'$ and $\tau_{\rM}'$. Define a new series $\{\tilde{\lambda}_i',0\leq i\leq t+1\}$, by putting $\tilde{\lambda}_{0}'=\tilde{\lambda}_{\rM}'$ and $\tilde{\lambda}_{t+1}'=\tau_{\rM}'$. There exists $0\leq i\leq t$, such that $\nu(\tilde{\lambda}_{i}')\in I_1$ but $\nu(\tilde{\lambda}_{i+1}')\in I_2$, where $\nu$ is defined as in Remark \ref{rem 0023}. Now we consider $\cP_{\rM'}:=\ind_{\tilde{J}_{\rM}'}^{\rM'}\cP_{\tilde{\lambda}_i'}\in\mathrm{Rep}_k(\rM')_{\mathcal{SC}(\tilde{\lambda}_i')}$ (see the beginning of Section \ref{section 0028} for the definition of $\mathcal{SC}(\tilde{\lambda}_i')$), hence $\cP_{\rM'}\in\mathrm{Rep}_k(\rM')_{[\pi_j'\sim]}$. Assume contrarily that $\cP_{\rM'}\cong P_1\oplus P_2$, where $P_1\in\mathrm{Rep}_k(\rM')_{I_1}$ and $P_2\in\mathrm{Rep}_k(\rM')_{I_2}$. Then $P_1\in\mathrm{Rep}_k(\rM')_{I_1\cap\mathcal{SC}(\tilde{\lambda}_i')}$ and $P_2\in\mathrm{Rep}_k(\rM')_{I_2\cap\mathcal{SC}(\tilde{\lambda}_i')}$. Since the union of $I_1\cap\mathcal{SC}(\tilde{\lambda}_i')$ and $I_2\cap\mathcal{SC}(\tilde{\lambda}_i')$ is a non-trivial disjoint union of $\mathcal{SC}(\tilde{\lambda}_i')$, the decomposition $\cP_{\rM'}\cong P_1\oplus P_2$ is contradicted with Proposition \ref{prop 0017}.

Secondly, we prove that $\mathrm{Rep}_k(\rM')_{[\rM,\pi]^{tw}}\cong\prod_{j=1}^s\mathrm{Rep}_k(\rM')_{[\pi_j',\sim]}$. We use the projective version in Remark \ref{rem 0029}. Now fix $j_0$, let $(\tilde{J}_{\rM}',\tilde{\lambda}_{j_0}')$ be a maximal simple supercuspidal $k$-type contained in $\pi_{j_0}'$, defined from a maximal simple supercuspidal $k$-type $(J_{\rM},\lambda_{\rM})$ of $\rM$. By Definition \ref{defn 0026} and Remark \ref{rem 0027}, we fix a maximal simple supercuspidal $k$-type for each $\rM'$-inertially equivalent supercuspidal classs contained in $[\pi_{j_0}',\sim]$, and denote by $\mathcal{I}_{j_0}$ the finite set of these maximal simple supercuspidal $k$-types. Define $\cP_{[\pi_{j_0}',\sim]}:=\oplus_{\tau'\in\mathcal{I}_{j_0}}\ind_{\tilde{J}_{\rM}'}^{\rM'}\cP_{\tau'}$ where $\cP_{\tau'}$ is the projective cover of $\tau'$. For each $1\leq j\leq s$ different from $j_0$, and let $[\pi_j',\sim]=\sqcup_{i=1}^t[\rM',\pi_{j,i}']_{\rM'}$ where $\pi_{j,i}'$ are irreducible supercuspidal and $t\in\mathbb{N}$. Fix a maximal simple supercuspidal $k$-type $(\tilde{J}_{j,i}',\tilde{\lambda}_{j,i}')$ contained in $\pi_{j,i}'$. Define $[\pi_{j_0},\sim]^{\perp}$ to be the union $\cup_{j\neq j_0}[\pi_j,\sim]$ and $\cP_{[\pi_{j_0},\sim]^{\perp}}:=\oplus_{j\neq j_0}\oplus_{i=1}^t\ind_{\tilde{J}_{j,i}'}^{\rM'}\cP_{\tilde{\lambda}_{j,i}'}$. We show that $\cP_{[\pi_{j_0},\sim]}$ and $\cP_{[\pi_{j_0},\sim]^{\perp}}$ verify the conditions in Remark \ref{rem 0029}. By Proposition \ref{rem 0016} and Lemma \ref{lem 0018}, we know that an irreducible sub-quotient of $\cP_{[\pi_{j_0},\sim]}$ belong to $[\pi_{j_0},\sim]$. Meanwhile an irreducible sub-quotient of $\cP_{[\pi_{j_0},\sim]^{\perp}}$ belong to $[\pi_{j_0},\sim]^{\perp}:=\cup_{j\neq j_0}[\pi_j,\sim]$. Condition $1$ and $3$ of Remark \ref{rem 0029} can be deduced from Proposition \ref{prop 008}. Condition $2$ of remark \ref{rem 0029} is verified from Remark \ref{rem 0027} that $``\sim"$ defines an equivalent relation, and $[\pi_{j_0}',\sim]$ is disjoint with $[\pi_{j_0},\sim]^{\perp}$. Hence by repeating the same operation on $\mathrm{Rep}_k(\rM')_{[\pi_{j_0},\sim]^{\perp}}$, and after finite times we obtain the desired decomposition.
\end{proof}

\begin{ex}
\label{ex 0029}
For $\rG'=\mathrm{SL}_n(F)$, when $\ell$ is positive, 
\begin{itemize}
\item it is not always true that the reduction modulo $\ell$ of a an irreducible $\ell$-adic supercuspidal representation of $\rG'$ is irreducible;
\item it is not always true that $\mathrm{Rep}_k(\rG')$ can be decomposed according to the $\rG'$-inertially equivalent supercuspidal classes as in Equation \ref{0001} in the case where $\ell=0$.
\end{itemize}
\end{ex}

\begin{proof}
Let $p=5,n=2,\ell=3$, and denote by $\overline{\rG}=\mathrm{GL}_2(\mathbb{F}_5)$ and by $\overline{\rG}'=\mathrm{SL}_2(\mathbb{F}_5)$. From \cite[\S 11.3.2]{Bonn} we know that there exists two irreducible supercuspidal $\overline{\mathbb{Q}}_{\ell}$-representations $\pi_1,\pi_2$ of $\overline{\rG}$ ($\pi_1$ corresponding to $-j^{\wedge}$ and $\pi_2$ corresponding to $\theta_0$ as in \cite[\S 11.3.2]{Bonn}), such that the reduction modulo $\ell$ of $\pi_1$ and $\pi_2$ are irreducible and coincide to each other. Meanwhile, the restriction $\pi_1\vert_{\overline{\rG}'}$ is irreducible but $\pi_2\vert_{\overline{\rG}'}$ is semisimple with length $2$. We denote by $\bar{\pi}_2$ the reduction modulo $\ell$ of $\bar{\pi}_2$. By \cite[\S 11.3.2]{Bonn} the length of $\bar{\pi}\vert_{\overline{\rG}'}$ is two, and denote by $\bar{\pi}_{2,1},\bar{\pi}_{2,2}$ the two irreducible direct components of  $\bar{\pi}_2\vert_{\overline{\rG}'}$ (in the notation of \cite[\S 11.3.2]{Bonn}, $\bar{\pi}_{2,1}$ and $\bar{\pi}_{2,2}$ corresponds to the reduction modulo $\ell$ of $R_{-}'(\theta_0)$ and $R_{+}'(\theta_0)$ respectively). In other words, the reduction modulo $\ell$ of the irreducible supercuspidal $\overline{\mathbb{Q}}_{\ell}$-representation $\pi_1\vert_{\overline{\rG}'}$ is reducible, and its Jordan-H$\ddot{\mathrm{o}}$lder components consist of $\bar{\pi}_{2,1}$ and $\bar{\pi}_{2,2}$. Both of $\bar{\pi}_{2,1}$ and $\bar{\pi}_{2,2}$ are supercuspidal by \cite[\S 3.2]{C2}, from the fact that their projective covers are cuspidal.

We consider the $\overline{\mathbb{Z}}_{\ell}$-projective cover $\mathcal{Y}_{\bar{\pi}_{2,1}}$ of $\bar{\pi}_{2,1}$. The strategy is to prove that the irreducible $\obQ_{\ell}$-representation $\pi_1\vert_{\rG'}$ is a subquotient of $\mathcal{Y}_{\bar{\pi}_{2,1}}\otimes\obQ_{\ell}$, from which we deduce that $\bar{\pi}_{2,2}$ is a subquotient of $\mathcal{Y}_{\bar{\pi}_{2,1}}\otimes k$, then we apply Proposition \ref{prop 0017}.

Let $\rU$ be the subgroup of upper triangular matrices of $\overline{\rG}$, the reduction modulo $\ell$ gives a bijection between non-degenerate $\overline{\mathbb{Q}}_{\ell}$-characters of $\rU$ and non-degenerate $k$-characters of $\rU$. Let $\theta_{\obQ_{\ell}}$ be a non-degenerate $\obQ_{\ell}$-character of $\rU$, and $\theta_{\ell}$ be the reduction modulo $\ell$ of $\theta_{\obQ_{\ell}}$, which is a non-degenerate $k$-character of $\rU$, such that $\bar{\pi}_{2,1}$ is generic according to $\theta_{\ell}$. By the unicity of Whittaker model, it follows that $\bar{\pi}_{2,2}$ is not generic according to $\theta_{\ell}$. By \cite{C2}, $\mathcal{Y}_{\bar{\pi}_{2,1}}\otimes\obQ_{\ell}$ is semisimple, and can be written as $\oplus_{s\in\mathcal{S}_{s_0}}\pi_{s,\theta_{\obQ_{\ell}}}$. Here $s_0$ is the $\ell'$-semisimple conjugacy class in $\overline{\rG}$ corresponding to $\pi_2$ by the theory of Deligne-Lusztig (or equivalently $s_0$ corresponds to $\theta_0$ under the notations of \cite[\S 11.3.2]{Bonn}), and $\mathcal{S}_{s_0}$ is the set of semisimple conjugay classes in $\overline{\rG}$ whose $\ell'$-part is equal to $s_0$. Denote by $\pi_{s}$ the irreducible supercuspidal $\obQ_{\ell}$-representation corresponding to $s$, and $\pi_{s,\theta}$ the unique irreducible component of $\pi_{s}\vert_{\overline{\rG}'}$ which is generic according to $\theta_{\obQ_{\ell}}$. Hence $\pi_1$ is a subrepresentation of $\mathcal{Y}_{\bar{\pi}_{2,1}}\otimes\obQ_{\ell}$, which implies that $\bar{\pi}_{2,2}$ is a sub-quotient of $\mathcal{Y}_{\bar{\pi}_{2,1}}\otimes k$, which is the $k$-projective cover of $\bar{\pi}_{2,1}$.

To go further to the $p$-adic groups, we conclude that the semisimplification of $\mathcal{Y}_{\pi_{2,1}}\otimes k$ consists with a non-trivial multiple of $\pi_{2,1}$ and a non-trivial multiple of $\pi_{2,2}$.

Now we consider the $p$-adic groups $\rG=\mathrm{GL}_2(F)$ and $\rG'=\mathrm{SL}_2(F)$, suppose that $F=\bQ_{5}$, and $k=\overline{\mathbb{F}}_3$. Let $J=\mathrm{GL}_2(\bZ_5)$, $J^1=1+\mathrm{M}_2(5\bZ_5)$ and $J'=J\cap\rG', J^{1'}=J^1\cap\rG'$. We have $J\slash J^1\cong\overline{\rG}$, and $J'\slash J^{1'}\cong\overline{\rG}'$. We still denote by $\pi_i,\bar{\pi}_i,\bar{\pi}_{2,i},i=1,2$ the corresponding inflation to $J'$ respectively. Hence $(J,\bar{\pi}_i),i=1,2$ are maximal simple supercuspidal $k$-types of $\rG$. According to \cite[3.18]{C1} and the fact that there are $4$ $\rG'$-conjugacy classes of non-degenerate characters on $\rU$, we deduce from the unicity of Whittaker models that for an irreducible cuspidal $k$-representation $\pi$ of $\rG$, the length of $\pi\vert_{\rG'}$ is a divisor of $4$, hence is prime to $5$. By Theorem 3.18 of \cite{C1}, the index $\vert \tilde{J}:J\vert$ is a $p$-power and a divisor of the length $\pi\vert_{\rG'}$, which implies that $\tilde{J}=J$. We deduce firstly that $(J',\pi_1\vert_{J'})$ is a maximal simple supercuspidal $\bar{\mathbb{Q}}_{\ell}$-type of $\rG'$, and $(J',\bar{\pi}_{2,i}), i=1,2$ are a maximal simple supercuspidal $k$-types of $\rG'$. Hence $\ind_{J'}^{\rG'}\pi_1\vert_{J'}$ is irreducible, but its reduction modulo $\ell$ had length two, with two factors $\ind_{J'}^{\rG'}\bar{\pi}_{2,i},i=1,2$, which is the first of this example. Secondly, we have that $(J',\bar{\pi}_{2,1})$ and $(J',\bar{\pi}_{2,2})$ are non $\rG'$-conjugate by the second part of Remark \ref{rem 0023}. By \cite[Propostion 2.35, Theorem 3.30]{C1},  $\Pi_1:=\ind_{J'}^{\rG'}\bar{\pi}_{2,1}$ and $\Pi_2:=\ind_{J'}^{\rG'}\bar{\pi}_{2,2}$ are different irreducible supercuspidal $k$-representations, and they defines different $\rG'$-inertially equivalent classes since there is no non-trivial $k$-characters on $\rG'$. The inflation of $\mathcal{Y}_{\bar{\pi}_{2,1}}$ to $J'$ is the $\overline{\bZ}_3$-projective cover of $\bar{\pi}_{2,1}$. By the previous paragraphs, $\bar{\pi}_{2,2}$ appears as a sub-quotient of $\mathcal{Y}_{\bar{\pi}_{2,1}}$. Apply Theorem \ref{thm 0030}, we conclude that two full sub-categories $\mathrm{Rep}_k(\rG')_{[\rG',\Pi_1]}$ and $\mathrm{Rep}_k(\rG')_{[\rG',\Pi_2]}$ belong to the same block.
\end{proof}

\end{document}